\documentclass[11pt]{article}

\usepackage[margin=1in]{geometry}
\usepackage{amsmath,amsfonts,amssymb,amsthm,amsopn,bm,mathtools}
\usepackage{graphicx}
\usepackage{algorithm}
\usepackage{algorithmic}
\usepackage{caption}
\usepackage{subfig}
\usepackage{cite}
\usepackage[hidelinks,bookmarksnumbered]{hyperref}
\usepackage[nameinlink,noabbrev]{cleveref}

\DeclareGraphicsExtensions{.pdf,.png,.jpg,.eps}
\allowdisplaybreaks
\numberwithin{equation}{section}
\numberwithin{algorithm}{section}

\newcommand{\email}[1]{\texttt{#1}}

\theoremstyle{plain}
\newtheorem{theorem}{Theorem}[section]
\newtheorem{lemma}[theorem]{Lemma}
\newtheorem{claim}[theorem]{Claim}
\newtheorem{fact}[theorem]{Fact}
\theoremstyle{definition}
\newtheorem{assumption}[theorem]{Assumption}

\crefname{theorem}{theorem}{theorems}
\Crefname{theorem}{Theorem}{Theorems}
\crefname{lemma}{lemma}{lemmas}
\Crefname{lemma}{Lemma}{Lemmas}
\crefname{claim}{claim}{claims}
\Crefname{claim}{Claim}{Claims}
\crefname{fact}{fact}{facts}
\Crefname{fact}{Fact}{Facts}
\crefname{assumption}{assumption}{assumptions}
\Crefname{assumption}{Assumption}{Assumptions}
\crefname{remark}{remark}{remarks}
\Crefname{remark}{Remark}{Remarks}
\crefname{hypothesis}{hypothesis}{hypotheses}
\Crefname{hypothesis}{Hypothesis}{Hypotheses}
\crefname{example}{example}{examples}
\Crefname{example}{Example}{Examples}
\crefname{algorithm}{algorithm}{algorithms}
\Crefname{algorithm}{Algorithm}{Algorithms}

\title{Adaptive Delayed-Update Cyclic Algorithm \\ for Variational Inequalities\thanks{\textbf{Funding:} Supported in part by the Air Force Office of Scientific Research under award number FA9550-24-1-0076 and by the U.S. Office of Naval Research under contract number N00014-22-1-2348. Any opinions, findings and conclusions or recommendations expressed in this material are those of the author(s) and do not necessarily reflect the views of the U.S. Department of Defense.}}

\author{
Yi Wei\thanks{Paul G. Allen School of Computer Science and Engineering, University of Washington, Seattle, WA (\email{ywei32@cs.washington.edu}). Much of this work was completed while the author was an undergraduate student at the University of Wisconsin--Madison.}
\and
Xufeng Cai\thanks{Department of Computer Sciences, University of Wisconsin--Madison, Madison, WI (\email{xcai74@wisc.edu}; \email{jelena@cs.wisc.edu}).}
\and
Jelena Diakonikolas\footnotemark[3]
}

\date{}
\hypersetup{
  pdftitle={Adaptive Delayed-Update Cyclic Algorithm for Variational Inequalities},
  pdfauthor={Yi Wei, Xufeng Cai, Jelena Diakonikolas}
}

\DeclareMathOperator{\Gap}{Gap}

\newcommand{\innp}[1]{\big\langle #1 \big\rangle}

\newcommand{\dom}{\mathrm{dom}}
\newcommand{\aduca}{\texttt{ADUCA}}
\newcommand{\graal}{\texttt{GRAAL}}
\newcommand{\coder}{\texttt{CODER}}
\newcommand{\pccm}{\texttt{PCCM}}

\def\1{\bm{1}}

\def\vb{{\bm{b}}}

\def\vu{{\mathbf{u}}}
\def\vv{{\mathbf{v}}}

\def\vx{{\mathbf{x}}}
\def\vy{{\mathbf{y}}}

\def\vzeta{{\bm{\zeta}}}

\def\Chat{{\hat{C}}}
\def\Lhat{{\hat{L}}}

\def\mA{{\bm{A}}}

\def\mF{{\bm{F}}}
\def\mFb{\bar{\bm{F}}}
\def\mFt{\tilde{\bm{F}}}

\def\mI{{\bm{I}}}

\def\mQ{{\bm{Q}}}

\def\mLambda{{\bm{\Lambda}}}

\DeclareMathAlphabet{\mathsfit}{\encodingdefault}{\sfdefault}{m}{sl}
\SetMathAlphabet{\mathsfit}{bold}{\encodingdefault}{\sfdefault}{bx}{n}

\def\gC{{\mathcal{C}}}

\def\gE{{\mathcal{E}}}

\def\gO{{\mathcal{O}}}

\def\gS{{\mathcal{S}}}
\def\gT{{\mathcal{T}}}

\def\sR{{\mathbb{R}}}

\newcommand{\R}{\mathbb{R}}

\newcommand{\cO}{\mathcal{O}}

\DeclareMathOperator*{\argmin}{argmin}

\newcommand{\littlesum}{\mathop{\textstyle\sum}}

\begin{document}

\maketitle

\begin{abstract}
Cyclic block coordinate methods are a fundamental class of first-order algorithms, widely used in practice for their simplicity and strong empirical performance. Yet, their theoretical behavior remains challenging to explain, and setting their step sizes---beyond classical coordinate descent for minimization---typically requires careful tuning or line-search machinery. 
In this work, we develop \aduca\ (Adaptive Delayed-Update Cyclic Algorithm), a cyclic algorithm addressing a broad class of Minty variational inequalities with monotone Lipschitz operators.  \aduca\ is parameter-free: it requires no global or block-wise Lipschitz constants and uses no per-epoch line search, except at initialization.  A key feature of the algorithm is using operator information delayed by a full cycle, which makes the algorithm  compatible with parallel and distributed implementations, and attractive due to weakened  synchronization requirements across blocks. We prove that \aduca\ attains (near) optimal global  oracle complexity as a function of target error $\epsilon >0,$ scaling with $1/\epsilon$ for monotone operators, or with $\log^2(1/\epsilon)$ for operators that are strongly monotone. 
\end{abstract}

\section{Introduction}\label{sec:intro}

Cyclic block coordinate methods have a long history in optimization and numerical analysis, due to their simplicity and often superior empirical performance compared to alternative approaches. Such methods partition the vector of variables into coordinate blocks (e.g., consisting of a single variable in the case of standard coordinate methods) and update them one variable block at a time, sweeping through all the coordinate blocks in one cycle of computation. Because such sequential strategies often lead to much simpler updates than updating the full vector in parallel, different instantiations of cyclic methods appear in various contexts throughout the history of optimization and numerical analysis; for instance, as the classical methods of Osborne---for matrix balancing \cite{osborne1960pre}, Kaczmarz---for solving linear systems \cite{karczmarz1937angenaherte}, or, more broadly, as general cyclic coordinate methods investigated since the  work of Ortega and Rheinboldt in the 1970s \cite{ortega2000iterative}. For certain problems like matrix balancing and fitting of linear models, cyclic methods are used as the default solvers in mainstream software packages \cite{mazumder2011sparsenet,friedman2010regularization,NumPyEig,JuliaEigen,MathWorksEig,smith2013matrix,RDocumentationBalance,MathWorksBalance}.     

Despite their long history and empirical success, cyclic methods are generally considered challenging to analyze and often exhibit worse theoretical guarantees than their randomized counterparts or even the full vector update baselines \cite{beck2013convergence,sun2019worst,SongDiakonikolas2023CODER}. Additionally, beyond exact cyclic descent for minimization problems, where the objective function is minimized along coordinate directions in each step, setting the step size of cyclic methods requires either a careful tuning or a backtracking line search \cite{SongDiakonikolas2023CODER,beck2013convergence,cai2023cyclic,lin2023accelerated,chow2017cyclic,gurbuzbalaban2017cyclic}. Outside minimization problems, and, in particular, in the settings of min-max optimization and general variational inequalities with monotone operators, exact cyclic descent is known not to be convergent in general (see, e.g., the discussion in \cite{SongDiakonikolas2023CODER}), so the choice of correct step sizes becomes even more important. 

In this work, we propose a new cyclic block coordinate method for generalized variational inequalities (see \Cref{sec:prob-setup} for relevant definitions) that we term Adaptive Delayed-Update Cyclic Algorithm (\aduca). \aduca\ is the first cyclic method addressing general classes of optimization problems that is both \emph{parameter-free} and \emph{locally-adaptive}: namely, the method does not require any knowledge of associated Lipschitz parameters of the problem and it can adjust the step size according to local geometry, without employing a line search\footnote{That is, beyond  employing a line search to set the initial step size. This is primarily a technical device to establish global convergence guarantees, as is standard in the literature; see, e.g.,~\cite{Alacaoglu2023BeyondGR, li2023simple}.}. Such methods are also referred to as being \emph{autoconditioned}; see, e.g., \cite{lan2024auto}. It is worth mentioning that effectively all autoconditioned methods analyzed in the literature perform full vector updates, possibly split over the primal and the dual in the case of min-max optimization; see, e.g.,  \cite{malitsky2020golden,Alacaoglu2023BeyondGR,lan2024auto}. 

A particularly interesting feature of \aduca\ is that it utilizes information delayed by a full cycle; this property is particularly useful for parallelizing updates and avoiding strict synchronization requirements of typical block coordinate methods. Moreover, this property can come at additional performance gains in parallel and distributed settings, as the algorithm updates can be carried over in parallel to communication needed for coordinating the updates (e.g., for communication between the distributed agents and a central server). 

\subsection{Problem Setup}\label{sec:prob-setup}

We consider generalized Minty variational inequality (GMVI) problems, which are of the form:
\begin{align} \label{eq:prob} \tag{P}
  \text{find }\vu^* \in \R^d \quad \text{such that} \quad \innp{\mF(\vu), \vu - \vu^*} + g(\vu) - g(\vu^{*})
  \ge 0, \, \quad \forall 
  \, \vu \in \R^d.
\end{align}
As is standard for such problems in block coordinate settings, we assume that $\mF: \R^d \to \R^d$ is a monotone, locally block-wise Lipschitz operator and $g: \R^d \to (-\infty, + \infty]$ is an extended-valued, proper, convex, lower semicontinuous, 
block-separable function,
with an efficiently computable (block) proximal operator. 

Given an error parameter $\epsilon > 0$, our goal is to find an approximate $\epsilon$-accurate solution to \eqref{eq:prob}, defined as $\vu_{\epsilon}^*$ satisfying:
\begin{equation}\label{eq:probapprox} 
    \text{Gap}(\vu_{\epsilon}^*; \vu) := \innp{\mF(\vu), \vu^*_{\epsilon} - \vu } + g(\vu^*_{\epsilon}) - g(\vu) \le \epsilon, \quad \forall \, \vu \in \R^d. \tag{P\textsubscript{approx}}
\end{equation}

$\text{Gap}(\vu_{\epsilon}^*; \vu)$ is a gap-like function commonly used in the analysis of methods for variational inequalities. By the problem definition \eqref{eq:prob}, it is guaranteed to be non-negative if $\vu_{\epsilon}^*$ (exactly) solves \eqref{eq:prob}. In other cases, it can be either positive, negative, or zero. In full generality, $\text{Gap}(\vu_{\epsilon}^*; \vu)$ is not guaranteed to be bounded, unless additional assumptions like strong convexity or compact domain are imposed on $g.$ In such cases, $\text{Gap}(\vu_{\epsilon}^*; \vu)$ can be bounded over a compact set $S$ containing iterates, following similar considerations as in \cite{nesterov2007}. For our results, $S$ can be chosen as the ball of diameter $c \|\vu_0 - \vu^*\|$ around a(ny) solution $\vu^*$, where $\vu_0$ is the initial point, and $c$ is some constant. This choice is justified by our analysis, which guarantees that all algorithm iterates remain in  $S.$ 

\subsection{Our Contributions}
\label{subsec:contributions}
Our main contributions are:

\textbf{A parameter-free, locally adaptive cyclic method for GMVIs.}
    To our knowledge, \aduca\ is the first cyclic block coordinate method  for GMVIs that is both \emph{parameter-free} and \emph{locally adaptive}: it requires no (global or block-wise) Lipschitz constants and uses no per-epoch line search beyond a one-time initialization. Beyond the special case of exact cyclic descent methods for minimization problems, it is moreover the first general block coordinate method (considered among either cyclic or randomized methods)  that is locally-adaptive and line search-free. This statement applies even when specialized to convex minimization problems. 

\textbf{Delayed-update mechanism for weak synchronization.}
    \aduca\ uses operator information delayed by one full cycle, which enables fully explicit cycle-wise step selection and reduces synchronization requirements across blocks, making the method naturally suitable for parallel/distributed implementations.

\textbf{New adaptive stepsize control tailored to cyclic updates.}
    We design an adaptive stepsize rule based on two computable local Lipschitz surrogates (capturing both consecutive-iterate variation and within-cycle variation), which is essential to control additional errors introduced by cyclic partial updates.

\textbf{Sharp convergence guarantees and global oracle complexity.}
    Under monotonicity and local block-wise Lipschitzness, we prove an $\mathcal{O}(1/\epsilon)$-type ergodic gap bound, matching the optimal rate in~\cite{Nemirovski2004MirrorProx,ouyang2021lower}.
    Under strong convexity (or a restricted strong monotonicity condition), we establish near-linear convergence, with $\log^2(1/\epsilon)$ iteration count scaling with the target error $\epsilon > 0$.

\subsection{Related Work}
\label{subsec:related-work}

Our work builds upon and is closely related to research on variational inequalities, cyclic block coordinate methods, and parameter-free methods. For brevity, we only review the literature most closely related to our work. Additional comparison to most closely related methods (\coder\ \cite{SongDiakonikolas2023CODER} and \graal\ \cite{malitsky2020golden,Alacaoglu2023BeyondGR}) is provided at the end of \Cref{sec:algorithm}, after our algorithm is formally stated.  

\paragraph{Variational inequalities}
For monotone VIs with Lipschitz continuous operators, classical results have established $\mathcal{O}(1/\epsilon)$ ergodic oracle complexity  \cite{Nemirovski2004MirrorProx,korpelevich1976extragradient,nesterov2007}, with improved linear convergence rate ($\mathcal{O}(\log(1/\epsilon)$) oracle complexity) under additional strong monotonicity \cite{nesterov2011solving}. Optimality of these results was established via oracle complexity lower bounds in \cite{ouyang2021lower}. More recent work has focused on locally adaptive methods \cite{malitsky2020golden,Alacaoglu2023BeyondGR}, last-iterate convergence and stationarity guarantees \cite{diakonikolas2020halpern,diakonikolas2021potential,golowich2020last}, and structured non-monotone operators \cite{diakonikolas2021efficient}, among other settings. 

\paragraph{Cyclic block-coordinate methods}
While cyclic block-coordinate methods are classical in optimization, nearly all existing theoretical results concern minimization problems; see e.g., \cite{beck2013convergence,sun2019worst,cai2023cyclic,lin2023accelerated,gurbuzbalaban2017cyclic}. To our knowledge, the only cyclic algorithm that applies to monotone VIs with Lipschitz operators (as in our setting) and attains optimal convergence rates is \coder, due to \cite{SongDiakonikolas2023CODER}. The same work introduced a block-Lipschitz model based on Mahalanobis seminorms that is also used in our work (see \Cref{sec:prelim}). While \cite{SongDiakonikolas2023CODER}  provided a parameter-free version of their method (\texttt{CODER-LineSearch}), this version of the method relies on a line search to determine the step size, and the theoretical results are in terms of global problem properties. 

\paragraph{Parameter-free methods via local smoothness}
A separate line of work develops \emph{parameter-free} first-order methods that avoid global Lipschitz information by estimating local smoothness / curvature between algorithm iterates.
In the VI setting, \cite{malitsky2020golden} introduced \graal: a method employing local Lipschitz estimation  to set stepsizes adaptively, without any per-iteration line search.
Later work \cite{Alacaoglu2023BeyondGR} refined and extended this approach; in particular, they showed that the explicit cap on the step size used in \cite{malitsky2020golden} can be removed, 
and established sublinear convergence under monotonicity as well as convergence guarantees for structured nonmonotone operators defined in \cite{diakonikolas2021efficient}.

Local-curvature adaptivity has also been extended in recent literature to other related optimization settings.
For example, for smooth convex minimization problems, \cite{malitsky2019adaptive} developed an adaptive gradient descent scheme, and subsequent work extended similar ideas to proximal methods~\cite{malitsky2024adaptive,latafat2024adaptive}.
Accelerated variants have been  developed in \cite{li2023simple,suh2025adaptive,borodich2026nesterov}.
In the saddle-point / primal--dual setting, \cite{vladarean2021first} provided a primal-dual first-order method adaptive to local smoothness, while \cite{lan2024auto} further extended such auto-conditioning properties to other classes of primal-dual algorithms.
Beyond standard convex settings, related curvature-adaptive principles have been applied to convex bilevel optimization~\cite{latafat2025convergence} and to nonconvex (weakly convex) and stochastic optimization via projected-gradient methods~\cite{lan2024projected}.

\paragraph{Parameter-free AdaGrad-type methods for VIs}
Finally, there is a separate class of adaptive step-size methods based on AdaGrad, which was originally developed for minimization problems \cite{duchi2011adaptive}. In the context of monotone VIs considered in our work, relevant work includes \cite{bach2019universal,DBLP:conf/aaai/EneNV21,DBLP:conf/aaai/EneN22}.  
The primary limitation of AdaGrad-type adaptivity in this line of work is  that the adaptive step size is non-increasing by design. As a result, such algorithms cannot exploit regions where larger step sizes may be admissible. In deterministic or low-noise regimes, this monotone decay can be overly conservative and lead to slower empirical convergence relative to methods that adapt to \emph{local} geometry via explicit local Lipschitz estimation, as in our work.
 \section{Preliminaries}\label{sec:prelim}
We consider the standard $d$-dimensional Euclidean vector space $(\sR^d, \|\cdot\|),$ where $\|\cdot\|$ denotes the $\ell_2$ norm and $\innp{\cdot, \cdot}$ denotes the inner product of vectors. 
For any positive integer $n$, we denote the set $\{1, 2, \dots, n\}$ by $[n]$. 
We denote $[x]_+ = \max\{0, x\}$ and define $\frac{x}{0} = \infty$. 
We reserve the notation $m$ for the number of blocks, and assume a disjoint partition of the coordinates $[d]$ into nonempty sets $\{\gS^i\}_{i \in [m]}$ of sizes $|\gS^i| = d_i$, where $\sum_{i = 1}^m d_i = d$. 
Without loss of generality, we assume that the partition is ordered, i.e.,\ for $1 \leq i < i' \leq m$, $\max_{j \in \gS^i} j < \min_{j' \in \gS^{i'}} j'$, as permutations of coordinates do not affect our results. 
We let $\vu^i$ and $\mF^i(\cdot)$ be the subvectors of $\vu$ and $\mF(\cdot)$ indexed by the coordinates in $\gS^i$, for $i \in [m]$. 

Given a matrix \(\mA\), we let \(\|\mA\| = \sup\{\|\mA \vu\| : \vu \in \sR^d, \|\vu\| \leq 1\}\) be the standard operator norm.
For a positive (semi)definite matrix $\mA$, $\|\cdot\|_{\mA} $ denotes the Mahalanobis (semi)norm defined by $\|\vu\|_{\mA} = \sqrt{\langle \mA \vu, \vu \rangle}$. 
We use $\mI_d$ to denote the identity matrix of size $d \times d$.
Following prior work on cyclic methods \cite{SongDiakonikolas2023CODER,lin2023accelerated,cai2023cyclic}, given positive semidefinite $d \times d$ matrices $\{\mQ^i\}_{i = 1}^m$, we define matrices $\widehat \mQ^i$ by
\begin{align*}
    (\widehat \mQ^i)_{j,k} =
    \begin{cases} 
    (\mQ^i)_{j,k}, & \text{if } \min\{j,k\} > \sum_{\ell=1}^{i-1} d_\ell, \\
    0, & \text{otherwise}.
    \end{cases}
\end{align*}
That is, $\widehat \mQ^i$ refers to $\mQ^i$ with the first $i - 1$ blocks of rows and columns set to zero.

 \subsection{Problem Setup}
We make the following standard assumptions.
\begin{assumption} \label{assp:monotone-sol}
The operator $\mF: \sR^d \rightarrow \sR^d$ is monotone,
i.e.,\ for any $\vu, \vv \in \sR^d$
\begin{align*}
    \innp{\mF(\vu) - \mF(\vv), \vu - \vv} \geq 0.
\end{align*}
The solution set of~\cref{eq:prob} is nonempty, i.e.,\ there exists at least one $\vu^*$ that solves~\cref{eq:prob}.
\end{assumption}

\begin{assumption}
\label{assp:gConvex}
  The function $g: \sR^d \to (-\infty,\infty]$ is proper, lower semicontinuous, and block-separable, i.e., $g(\vu) = \sum_{i=1}^{m} g^{i}(\vu^{i})$. Given positive diagonal matrices $\{\mLambda_i\}_{i \in [m]}$, where $\mLambda_i \in \R^{d_i \times d_i}$, each $g^i$ is $\mu$-strongly convex with a \emph{possibly unknown} convexity modulus $\mu \geq 0$, 
  i.e.,\ for any $\vu^i, \vv^i$ and any subgradient $\vzeta^i \in \partial g^i(\vv^i)$
  \begin{align*}
      g^i(\vu^i) \geq g^i(\vv^i) + \innp{\vzeta^i, \vu^i - \vv^i} + \frac{\mu}{2} \big\| \vu^i - \vv^i \big\|^2_{\mLambda_i}, 
  \end{align*}
  and admits an efficiently computable proximal operator with parameter $\eta > 0$, i.e.,\
  \begin{align*}
      \mathrm{prox}_{\eta g^i}(\vv^i; \mLambda_i) := \argmin_{\vu^i} \Big\{ g^i(\vu^i) + \frac{1}{2 \eta}\| \vu^i - \vv^i \|^2_{\mLambda_i} \Big\}.
  \end{align*}
\end{assumption}
In Assumption~\ref{assp:gConvex}, for notational convenience, we use  $\mu = 0$ for problems in which $g$ is only convex and $\mF$ is only monotone.  
The role of matrices $\mLambda_i$ will become clear from the definition of local block-wise Lipschitzness of $\mF$; namely, these matrices allow for rescaling of the problem along the coordinate blocks. If matrices $\mLambda_i$ are all identity matrices of appropriate sizes, then there is no rescaling and each norm $\|\cdot\|_{\mLambda_i}$ is the simple $\ell_2$ norm. The (strong) convexity and `prox-friendly' assumptions are stated with respect to $\|\cdot\|_{\mLambda_i}$, which are equivalent to their standard $\ell_2$ counterparts up to the scaling of $\mu$ and $\eta$. Throughout, we use $\mLambda$ to denote the diagonal matrix obtained by concatenating the diagonals of $\mLambda_1, \dots, \mLambda_m$ on the main diagonal, so that for any $\vu \in \sR^d,$ we have $\|\vu\|_{\mLambda}^2 = \sum_{i=1}^m \|\vu^i\|_{\mLambda_i}^2$.

Finally, we assume that block operators $\mF^i$ are \emph{locally Lipschitz}, as stated below. 

\begin{assumption}
\label{assp:localLipschitz}
  Given positive diagonal matrices $\{\mLambda_i\}_{i \in [m]},$ for every compact set $\gC \subseteq \sR^d,$ there exist positive semidefinite matrices $\{\mQ_{\gC}^i\}_{i \in [m]}$ such that each $\mF^{i}(\cdot)$ is $1$-Lipschitz continuous with respect to $\|\cdot\|_{\mQ_\gC^i}$, i.e., for any $\vu, \vv \in \gC$ 
  \begin{equation}\label{eq:lipschitzLocal}
    \|\mF^i(\vu) - \mF^i(\vv)\|_{\mLambda_{i}^{-1}} \leq \| \vu - \vv \|_{\mQ_{\gC}^i}.
  \end{equation}
\end{assumption}
Assumption~\ref{assp:localLipschitz} is a \emph{local} block-wise Lipschitz condition that was previously used as a \emph{global} condition to study fine-grained convergence properties of cyclic methods  \cite{SongDiakonikolas2023CODER,cai2023cyclic,lin2023accelerated}. We emphasize here that matrices $\mQ^i_\gC$ from Assumption \ref{assp:localLipschitz} need not be known to the algorithm---instead, the algorithm's convergence guarantee is w.r.t.\ the best possible such choice of matrices $\mQ^i_\gC$ based on all iterates. Further, existence of such matrices is guaranteed given local Lipschitzness of $\mF,$ since we can take $\mLambda_i = \mI_{d_i}$ and $\mQ_\gC^i = L_\gC \mI_d$ for the local Lipschitz constant $L_\gC > 0$ of the full operator $\mF.$ Similarly, this condition is implied (and most general) among block-wise Lipschitz conditions for variational inequalities studied in the literature. In particular, it is implied by the (local versions of) block-wise Lipschitz assumptions of the form $\|\mF^i(\vu) - \mF^i(\vv)\| \leq L_\gC^i\|\vu - \vv\|$ from prior work on block coordinate methods for VIs \cite{kotsalis2022simple,diakonikolas2025block}, by taking $\mLambda_i = \mI_{d_i}$ and $\mQ_\gC^i = L_\gC^i \mI_d$ for some $L_\gC^i > 0$. Finally, we remark that block-wise Lipschitz assumptions restricted to vectors that differ only on block $i$ common to the literature on block coordinate methods for minimization problems \cite{nesterov2012efficiency,beck2013convergence} are incompatible with general VIs without additional assumptions on $\mF$ (such as block cocoercivity \cite{chow2017cyclic} or uniform boundedness of block operator norms $\|\mF^i\|$ \cite{yousefian2018stochastic}), as remarked in \cite{SongDiakonikolas2023CODER,diakonikolas2025block}. 

\subsection{Review of Standard Facts}

Finally, we list a few standard facts and definitions that are useful to our analysis. 

\begin{fact}\label{fact:reverse-of-young}
    For any $x, y \in \sR$ and any $\alpha > 0,$ we have $(x + y)^2 \geq \frac{\alpha}{1+\alpha} x^2  - \alpha y^2.$
\end{fact}
\begin{proof}
    This is a simple consequence of Young's inequality, by which for any $a, b \in \sR$ and any $\alpha > 0,$ $(a + b)^2 \leq \big( 1+ \frac{1}{\alpha})a^2 + (1+ \alpha)b^2,$
by taking $a = x + y$, $b = - y.$ 
\end{proof}

\begin{fact}\label{fact:quadratic-cvx-comb}
    For any Euclidean norm $\|\cdot\|$, any  $\vx, \vy \in \sR^d$, and any $\alpha > 0,$
\begin{equation}\notag
        \big\| \alpha \vx + (1-\alpha)\vy \big\|^2 = \alpha\|\vx\|^2 + (1-\alpha)\|\vy\|^2 -\alpha(1-\alpha)\|\vx - \vy\|^2. 
    \end{equation}
\end{fact}
The proof of this fact is standard and omitted for brevity.

The following fact is a standard consequence of mirror descent-style updates, provided for completeness.
\begin{fact}\label{fact:update-3PI}
    Let $g:\sR^d \to \sR$ be $\mu$-strongly convex w.r.t.\ a norm $\|\cdot\|_{\mLambda}$ with $\mu \geq 0$, let $a > 0$, and let $\mFb, \vu \in \sR^d$. Define \begin{equation}\notag
        \vu^+ = \argmin_{\vv \in \sR^d}\big\{a\innp{\mFb, \vv} + a g(\vv) + \frac{1}{2}\|\vv - \vu\|_{\mLambda}^2\big\}.
    \end{equation}
Then, for any $\vv \in \sR^d,$
\begin{equation}\notag
        a \big(\innp{\mFb, \vu^+ - \vv} + g(\vu^+) - g(\vv)\big) \leq \frac{1}{2}\|\vv - \vu\|_{\mLambda}^2 - \frac{1}{2}\|\vu^+ - \vu\|_{\mLambda}^2 - \frac{a \mu + 1}{2}\| \vv - \vu^+\|_\mLambda^2. 
    \end{equation}
\end{fact}
\begin{proof}
    Let $M(\vv) := a\innp{\mFb, \vv} + a g(\vv) + \frac{1}{2}\|\vv - \vu\|_{\mLambda}^2$ so that $\vu^+ = \argmin_{\vv \in \sR^d} M(\vv).$ Since $M$ is $(a \mu + 1)$-strongly convex w.r.t.\ $\|\cdot\|_{\mLambda}$ and  it is minimized by $\vu^+,$ we have that for any $\vv \in \sR^d,$
\begin{equation}\notag
        M(\vv) \geq M(\vu^+) + \frac{a \mu + 1}{2}\|\vv - \vu^+\|_{\mLambda}^2.
    \end{equation}
Plugging in the definition of $M$ and rearranging leads to the claimed inequality.
    \end{proof}
Finally, to bound the restricted gap function $\Gap(\vu; \vv)$ defined in \eqref{eq:probapprox}, we will sequentially bound the individual ``merit'' functions $h_k$ defined by 
\begin{align}\label{eq:potential}
  h_k(\vu) := a_k \big(\innp{\mF(\vu_k), \vu_k - \vu} + g(\vu_k) - g(\vu)\big)
\end{align}
for $\{\vu_k\}_{k \geq 1}$, as is standard for the analysis of VI methods (see, e.g., \cite{Nemirovski2004MirrorProx}). 

In our analysis, we use auxiliary notation $\{\theta_k\}_{k \geq 1}$, which act as weights. These are defined via the accumulated product of $\{\omega_k\}_{k \geq 0}$, where $\omega_k \in (0, 1]$ for all $k \geq 0$:
\begin{align}\label{eq:theta-def}
    \theta_k = 1 / \prod_{i = 0}^{k - 1}\omega_i \; (k \geq 1), \quad \text{ equivalently, } \quad 
    \theta_{k} = \theta_{k - 1} / \omega_{k - 1} \; (k \geq 1, \, \theta_0 = 1).
\end{align}
For completeness, we include the following fact that relates the merit functions to the restricted gap function.
\begin{fact}\label{lemma:gap}
Suppose Assumptions~\ref{assp:monotone-sol}~and~\ref{assp:gConvex} hold. 
Given $K \ge 1$ and $\epsilon > 0$, 
for any sequence of vectors $\{\vu_k\}_{k=1}^K$ in $\sR^d$ and any sequences of positive numbers $\{ a_k\}_{k=1}^K$ and $\{ \theta_k \}_{k=1}^K$, if \,$\frac{\sum_{k=1}^K \theta_k h_k(\vu)}{{\sum_{k=1}^K \theta_k a_k}} \leq \epsilon$
for $\vu \in \dom(g)$, then $\mathrm{Gap}(\hat \vu_K; \vu) \leq \frac{\sum_{k=1}^K  \theta_k h_k(\vu)}{\sum_{k=1}^K \theta_k a_k} \leq \epsilon$, 
where $\hat \vu_K = \frac{\sum_{k=1}^K \theta_k a_k \vu_k}{\sum_{k=1}^K \theta_k a_k}$.
\end{fact}
\begin{proof}
By the convexity of $g$ and the monotonicity of $\mF$, 
\begin{align*}
    \;& \innp{\mF(\vu), \hat{\vu}_K - \vu} + g(\hat{\vu}_K) - g(\vu) \\
    \leq \;& \frac{1}{\sum_{k=1}^K \theta_k a_k} \sum_{k=1}^K \theta_k a_k \Big(\innp{\mF(\vu), \vu_k - \vu} + g(\vu_k) - g(\vu)\Big) \\
    \leq \;& \frac{1}{\sum_{k=1}^K \theta_k a_k} \sum_{k=1}^K \theta_k a_k \Big(\innp{\mF(\vu_k), \vu_k - \vu} + g(\vu_k) - g(\vu)\Big) 
    =  \frac{\sum_{k=1}^K  \theta_k h_k(\vu)}{\sum_{k=1}^K \theta_k a_k} \leq \epsilon.
\end{align*}
\end{proof}

 \section{\aduca: Adaptive Delayed-Update Cyclic Algorithm}\label{sec:algorithm}

\begin{algorithm}[t]
  \caption{\aduca: Adaptive Delayed-Update Cyclic Algorithm for VI} 
  \label{alg:ADUCA}
  \begin{algorithmic}[1]
    \STATE \textbf{Input:} 
    $m$, $K$, $\vu_0 = \vv_0$, $\beta \in (\frac{\sqrt{5} - 1}{2}, 1)$, $\gamma \in \big(0, 1- {\textstyle \frac{1}{\beta(1+\beta)}} \big)$, $\rho \in \big(1, \frac{1}{\beta}\big)$, $\mLambda$, $\mu \geq 0$
    \STATE \textbf{Initialize:} 
    $\theta_0 = \omega_0 = 1$. Initialize $a_0 = a_{-1}$, $\tilde \mF_0 = \mF(\vu_0)$ using \cref{alg:init-robust}
\STATE $\vu_{1} = \argmin_{\vu}\big\{a_0\innp{\mF(\vu_0), \vu} + a_0 g(\vu) + \frac{1}{2}\|\vu - \vv_0\|^2_{\mLambda} \big\}$
\STATE $\tilde \mF_{1}^{i} = \mF^{i}(\vu^{1}_{1}, \ldots , \vu^{i - 1}_{1}, \vu^{i}_{0}, \ldots , \vu^{m}_{0})$ $(i \in [m])$
\FOR{$k=1$ \TO $K - 1$}

\STATE Set $a_k$ using either the complete step size conditions \eqref{eq:complete-step-size-conditions} or one of the simplified step size conditions \eqref{eq:simple-step-size-conditions-known-mu}, \eqref{eq:simple-step-size-conditions-unknown-mu}, or \eqref{eq:fully-spec-step-size-exp}
\FOR{$i = 1$ \TO $m$}
\STATE $\bar{\mF}^{i}_{k} = \tilde \mF^{i}_{k} 
          + \frac{a_{k-1}\omega_{k - 1}}{a_{k}}\big( \mF^{i}(\vu_{k-1}) - \tilde \mF^{i}_{k-1} \big)$
\STATE $\vv_{k}^{i} = (1 - \beta)\vu_{k}^{i}  + \beta \vv_{k-1}^{i}$
\STATE $\vu^{i}_{k + 1} =  \argmin_{\vu^i}\big\{a_k\innp{ \bar \mF_k^i, \vu^i} + a_kg^i(\vu^i) + \frac{1}{2}\|\vu^i - \vv_k^i\|^2_{\mLambda_i} \big\}$
\STATE $\tilde \mF^{i}_{k+1} = \mF^{i}(\vu^{1}_{k+1}, \ldots , \vu^{i - 1}_{k+1}, \vu^{i}_{k}, \ldots , \vu^{m}_{k})$
        \STATE $\theta_k = \theta_{k-1}/\omega_{k-1}$, $\omega_{k} = \frac{1 + \rho \beta \mu a_{k}}{1 + \mu a_{k}}$
\ENDFOR
    \ENDFOR
    \RETURN $\hat{\vu}_{K} = \sum_{k = 1}^{K} \theta_k a_k \vu_{k} / \sum_{k = 1}^K \theta_k a_k$ or $\vv_{K+1}$
  \end{algorithmic}
\end{algorithm}

Our algorithm is formally stated in~\cref{alg:ADUCA}, and its components---the update scheme and the adaptive step size scheme---are introduced in detail below. 

\subsection{Update scheme}
Our proposed algorithm utilizes cyclic block coordinate updates with extrapolation. For the $i^{\mathrm{th}}$ inner iteration of the $k^{\mathrm{th}}$ cycle, let $\tilde \mF^{i}_{k} := \mF^{i}(\vu^{1}_{k}, \dots ,\vu^{i - 1}_{k}, \allowbreak \vu^{i}_{k - 1}, \dots, \vu^{m}_{k - 1})$ denote the operator of the partially updated and one cycle-delayed iterate $\vu_{k - 1, i} := (\vu^{1}_{k}, \dots ,\vu^{i - 1}_{k}, \allowbreak \vu^{i}_{k - 1}, \dots, \vu^{m}_{k - 1})^\top$. Then, \aduca\ computes the new block iterate $\vu_{k + 1}^i$ by 
\begin{align}
  \bar{\mF}^{i}_{k} = \;& \tilde \mF^{i}_{k} + \frac{a_{k-1}\omega_{k - 1}}{a_{k}}\big(\mF^{i}(\vu_{k - 1}) - \tilde \mF^{i}_{k - 1}\big), \label{eq:barF} \\
  \vv_{k}^{i} = \;& (1 - \beta)\vu_{k}^{i} + \beta\vv_{k - 1}^{i}, \label{eq:vk} \\
  \vu^{i}_{k + 1} = \;& \argmin_{\vu^i}\big\{a_k\innp{ \bar \mF_k^i, \vu^i} + a_k g^i(\vu^i) + \frac{1}{2}\|\vu^i - \vv_k^i\|^2_{\mLambda_i} \big\}, \label{eq:uk}
\end{align}
where $\{a_k\}_{k \geq 0}$ are the adaptive step sizes to be set below, and $\{\omega_k\}_{k \geq 0}$ are introduced to account for possible strong convexity and achieve faster near-linear convergence if it applies.
Note that the minimization step~\cref{eq:uk} is performed with respect to $\|\cdot\|_{\mLambda_i}$, allowing block-wise scaling for the step sizes and complementing cycle-wise adaptiveness. 
Additionally, different from previous block coordinate methods, the algorithm uses the operator values $\mFt_k$ rather than $\mFt_{k + 1}$, which are one cycle \emph{delayed} (hence `\emph{Delayed}' in \aduca), to update $\vu_{k + 1}$ in~\eqref{eq:uk}. This arises from the need to adaptively determine the cycle-wise step sizes at the beginning of each cycle. 
This delayed scheme also yields an appealing feature of \aduca: despite the inherently sequential nature of cyclic updates, the updates can in fact be implemented in parallel. The delay has further implications on practical implementations, as strict synchronization of updates is not required. In Sections~\ref{sec:analysis}~and~\ref{sec:num-exp}, we show that the algorithm exhibits efficient convergence, both theoretically and empirically, despite the delayed updates. 

\subsection{Adaptive step size}
Before introducing the step size scheme, it is helpful to first  define the hyperparameters used by the algorithm. We discuss possible specific choices later in this section. We define 
\begin{align}\label{eq:hyperparameters}
    \beta \in \big( {\textstyle \frac{\sqrt{5} - 1}{2}}, 1 \big), \quad \gamma \in \big(0, 1- {\textstyle \frac{1}{\beta(1+\beta)}} \big), \quad \rho \in \big( 1, {\textstyle \frac{1}{\beta}} \big).
\end{align}
Here, $\beta$ denotes the extrapolation weight in~\cref{eq:vk}, while $\gamma$ and $\rho$, dependent on $\beta$, bound the step size growth. We further define auxiliary notation $\eta := \sqrt{\frac{\gamma(1+\beta)}{1+\beta^2}}$. In the following, we provide a complete set of conditions for the step size update, which can be employed to obtain potentially larger step sizes when the strong convexity parameter $\mu$ is known and positive (e.g., when $g$ is used as an explicit strongly convex regularizer). This set of conditions is also more convenient for carrying out the analysis, as it was, in fact, obtained from the analysis. We then also provide a set of sufficient conditions for the step size update, which are simpler and do not require knowledge of any problem parameters (i.e., do not require knowing $\mu$).

Our step size conditions depend on \emph{local} Lipschitz constants, defined by
\begin{equation}\label{eq:local-Lip-constants}
    \hat{L}_k := \frac{\| \mF(\vu_k) - \tilde \mF_k \|_{\mLambda^{-1}}}{\|\vu_k - \vu_{k-1}\|_{\mLambda}}, \quad 
    L_k := \frac{\| \mF(\vu_k) - \mF(\vu_{{k-1}}) \|_{\mLambda^{-1}}}{\|\vu_k - \vu_{k-1}\|_{\mLambda}}
\end{equation}
The boundedness of $L_k$ follows directly from the standard definition of Lipschitz continuity, with additional flexibility from block-wise scaling via $\mLambda$.
For $\hat{L}_k$, 
by~\cref{assp:localLipschitz}, there exist $\{\mQ_\gC^i\}_{i \in [m]}$\footnote{Here, $\gC$ is a compact set containing the algorithm's iterates, whose existence follows from the boundedness of the iterates in Theorem~\ref{theorem:bound-gap} (Section~\ref{sec:analysis}).} and corresponding $\{\widehat \mQ_\gC^i\}_{i \in [m]}$ such that
\begin{align*}
    \| \mF(\vu_k) - \tilde \mF_k \|_{\mLambda^{-1}}^2 
    \leq \;& \littlesum_{i = 1}^m \| \vu_k - \vu_{k - 1, i} \|_{\mQ_\gC^{i}}^2 
    = \littlesum_{i = 1}^m \| \vu_k - \vu_{k - 1} \|_{\widehat \mQ_\gC^{i}}^2 \\
\leq \;& \big\|\mLambda^{-1/2} \big(\littlesum_{i = 1}^m \widehat \mQ_\gC^{i}\big) \mLambda^{-1/2}\big\| \|\vu_k - \vu_{k - 1}\|_\mLambda^2.
\end{align*}
Thus, in each cycle, $\hat L_k$ locally estimates $\hat L := \big\|\mLambda^{-1/2} \big(\sum_{i = 1}^m \widehat \mQ_\gC^{i}\big) \mLambda^{-1/2}\big\|^{1/2}$, a fine-grained Lipschitz constant characterizing the aggregated progress of cyclic updates~\cite{SongDiakonikolas2023CODER,cai2023cyclic,song2022coordinate,cai2023empirical}. For further details, we refer to the discussion in~\cite{SongDiakonikolas2023CODER}.

\paragraph{A complete set of step size conditions} For $k \geq 0,$ define 
\begin{equation}\label{eq:omega-def}
    \omega_{k} := \frac{1 + \rho \beta \mu a_{k}}{1 + \mu a_{k}}\;\;\text{ and }\;\;\phi_k := \sqrt{\frac{a_k (\beta \mu a_{k-1} + 1)}{\beta a_{k-1}}}. 
\end{equation} 
Let $\tau \in (0, 1)$. A complete set of step size conditions that arise from our analysis and are together sufficient for guaranteeing the claimed algorithm convergence is:
\begin{equation}\label{eq:complete-step-size-conditions}
    \begin{aligned}
        a_k &\leq \min\Big\{\rho a_{k-1},\, \frac{\beta(1+\beta)(1-\gamma)a_{k-1}}{[1- (1-\gamma)\beta^2\mu a_{k-1}]_+},\,  \\
        &\quad\quad\quad\quad\frac{\eta}{2}\min
        \Big\{ 
        \frac{\sqrt{\tau}\, \omega_{k-1}}{\sqrt{3}L_k\sqrt{1+\omega_{k-1}}},\, \frac{\sqrt{(1-\tau)\omega_{k-1}/2}}{\hat{L}_k}
        \Big\}
        \phi_{k-1}\Big\},
        &&\text{ for } k \geq 1,\\
        a_k &\leq \frac{\eta \sqrt{\tau}}{2\sqrt{3}}\frac{\omega_{k-1}\sqrt{\omega_{k-2}}}{\hat{L}_{k-1}\sqrt{1 + \omega_{k-1}}}\phi_{k-2}, &&\text{ for } k \geq 2,\\
        a_k &\leq \frac{\eta\sqrt{\tau}}{2\sqrt{3}}\frac{\omega_{k-1}}{\hat{L}_{k-2}\sqrt{1 + \omega_{k-1}}}\sqrt{\frac{\omega_{k-3}}{\omega_{k-2}}}\frac{a_{k-1}}{a_{k-2}}\phi_{k-3}, &&\text{ for } k \geq 3.
    \end{aligned}
\end{equation}
The conditions on the hyperparameters from \eqref{eq:hyperparameters} ensure that the step size can be increased between iterations, based on the first two conditions in \eqref{eq:complete-step-size-conditions}. The remaining conditions are dependent on the local Lipschitz properties of the monotone operator~$\mF.$ 

\paragraph{Simplified step size conditions}
The step sizes in \eqref{eq:complete-step-size-conditions} can be simplified by (i) defining $\rho_0 := \min\{\rho, \beta(1+\beta)(1-\gamma)\}$ so the first two conditions are implied by $a_k \leq \rho_0 a_{k-1}$ and (ii) arguing that the conditions from the last two inequalities (for $k \geq 2$ and $k \geq 3$) are implied by the first set of inequalities for an appropriate choice of $\tau = \tau(\rho_0)$. We summarize such simplified step sizes in the lemma below, whose proof is deferred to the appendix.

\begin{lemma}[Simplified step sizes]\label{lemma:simplified-steps}
    Let $\rho_0 := \min\{\rho, \beta(1+\beta)(1-\gamma)\}$ and $\tau = \frac{3\rho_0^2(1 + \rho\beta)}{2(\rho\beta)^2 + 3\rho_0^2(1 + \rho\beta)} \in (0, 1).$ Then, either of the step size conditions for $k \geq 1$ stated in \eqref{eq:simple-step-size-conditions-known-mu} and \eqref{eq:simple-step-size-conditions-unknown-mu} implies the complete step size conditions in \eqref{eq:complete-step-size-conditions}. 
\begin{align}
        a_k &\leq \min\Big\{\rho_0 a_{k-1},\,  \frac{\eta}{2}\min\Big\{ \Big( \frac{\sqrt{\tau}\, \omega_{k-1}}{\sqrt{3}L_k\sqrt{1+\omega_{k-1}}},\, \frac{\sqrt{(1-\tau)\omega_{k-1}}}{\sqrt{2}\hat{L}_k}\Big)\Big\}\phi_{k-1}\Big\},\label{eq:simple-step-size-conditions-known-mu}\\
        a_k &\leq \min\Big\{\rho_0 a_{k-1},\,  \frac{\eta}{2\sqrt{\beta}}\min\Big\{ \Big( \frac{\sqrt{\tau}\, \rho\beta}{\sqrt{3}L_k\sqrt{1+\rho\beta}},\, \frac{\sqrt{(1-\tau)\rho\beta}}{\sqrt{2}\hat{L}_k}\Big)\Big\}\sqrt{\frac{a_{k-1}}{a_{k-2}}}\Big\}.\label{eq:simple-step-size-conditions-unknown-mu}
    \end{align}
\end{lemma}

For example, one could choose $\rho = 1.2, \, \beta=0.8, \, \gamma=0.2$ for the algorithm parameters (which can be verified to satisfy all the stated constraints; observe that $\tau$ and $\eta$ are computable based on these parameters) and obtain fully specified step size update that does not require any further parameter selection. Such a choice leads to the following simple and fully specified step size rule:
\begin{equation}\label{eq:fully-spec-step-size-exp}
    a_k = \min\Big\{1.15\,a_{k-1},\, \min\Big\{\frac{0.093}{L_k},\, \frac{0.079}{\hat{L}_k}\Big\}\sqrt{\frac{a_{k-1}}{a_{k-2}}}\Big\}. 
\end{equation}

\subsection{Additional Remarks}

Even though the step sizes can be made fully independent of any global problem parameters (including $\mu$; see \eqref{eq:fully-spec-step-size-exp}), the algorithm itself, as specified, requires knowledge of $\mu$ to update $\omega_k$, which appears in the update for the extrapolated operator $\mFb_k$. This might appear to undermine the adaptivity of \aduca\ when $\mu>0$. However, $\{\omega_k\}_{k \geq 0}$ was introduced only to facilitate the near-linear rate convergence analysis and is not essential to the practical implementation. In general, since $\mu$ is used solely to certify a near-linear rate, it can be replaced throughout the analysis (and hence in the definition of $\omega_k$) by any conservative lower bound $\underline\mu\in(0,\mu]$, in which case the strong convexity assumption~\cref{assp:gConvex} and~\cref{fact:update-3PI} remain valid with $\underline\mu$. Using such a lower bound may weaken the theoretical rate constant, but the effect is  mild because the factor $\frac{1+\rho\beta\mu a_k}{1+\mu a_k}$ always satisfies $\rho\beta<\frac{1+\rho\beta\mu a_k}{1+\mu a_k}<1$ for $\mu a_k>0$, and in typical parameter choices $\rho\beta$ is close to $1$; for instance, the values used in our experiments, $\rho=1.2$ and $\beta=0.8$, yield $\rho\beta=0.96$. The empirical impact of using a conservative lower bound is further examined in the ablation study in~\Cref{subsec:ablation}, which shows that the algorithm performance is insensitive to the choice of $\mu;$ even setting $\mu = 0$ still leads to convergence that is indistinguishable from runs with higher values of $\mu$. We conjecture that this is not a coincidence and that the algorithm can, in fact, be proven to converge linearly even if $\mu$ is set to zero. This would however come at a cost of higher complexity in an  already quite technical analysis, thus we defer such considerations to future work.  

\begin{figure}[t]
    \hspace*{\fill}\subfloat[\texttt{a9a}]{\includegraphics[width=0.33\textwidth]{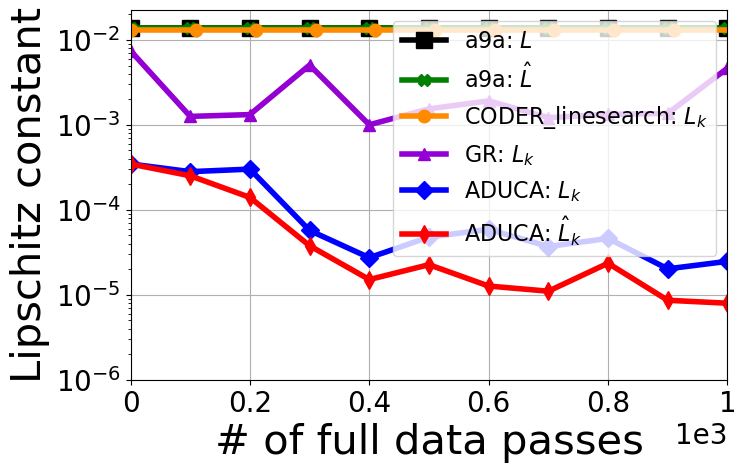}\label{fig:L_a9a}}\hfill
    \subfloat[\texttt{gisette}]{\includegraphics[width=0.33\textwidth]{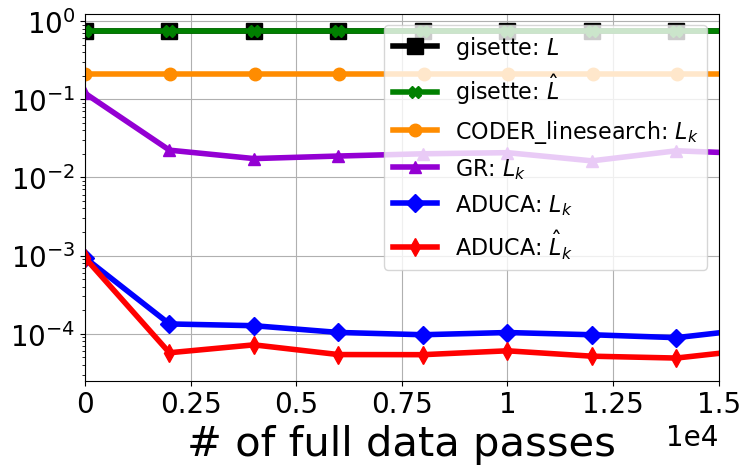}\label{fig:L_gisette}}
    \hspace*{\fill}
    \subfloat[\texttt{SUSY}]{\includegraphics[width=0.33\textwidth]{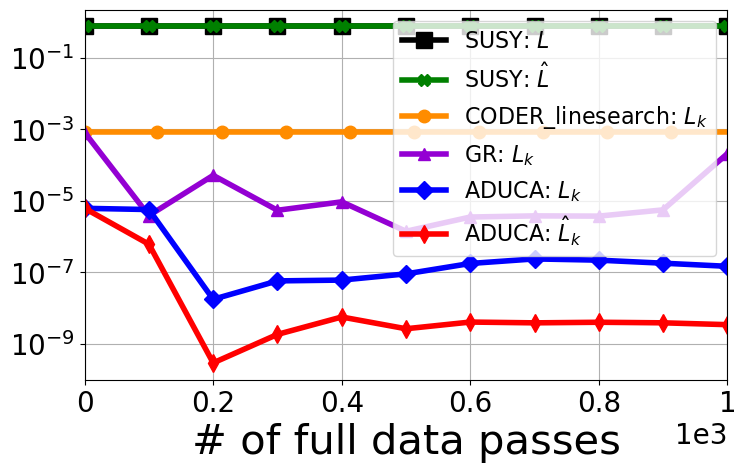}\label{fig:L_SUSY}}
    \hspace*{\fill}
    \caption{Comparisons of global Lipschitz constants and local Lipschitz estimates used by different algorithms, on the Support Vector Machine problem with datasets \texttt{a9a}, \texttt{gisette}, and \texttt{SUSY}. See Section~\ref{subsec:svm} for further details. }
    \label{fig:L}
\end{figure}

We now provide illustrative numerical results to provide intuition on why adaptive estimates $\hat L_k$ and $L_k$ can lead to faster convergence of \aduca\ compared to other related methods. In~\cref{fig:L}, we compare different Lipschitz estimates on the regularized support vector machine problem; see details in~\cref{sec:num-exp}. Here, $L$ is the standard, global Lipschitz constant, $\hat{L}$ and $\{ L_k \}_{k \ge 1}$ of \texttt{CODER-LineSearch} are the global block-wise Lipschitz constant w.r.t.\ Mahalanobis norm and the estimated Lipschitz constants via line search, respectively, defined in~\cite{SongDiakonikolas2023CODER}; $\{ L_k \}_{k \ge 1}$ of \graal\ is the local Lipschitz estimation w.r.t.\ standard Euclidean norm defined in~\cite{malitsky2020golden}. On all three datasets, $\hat L_k$ and $L_k$ constructed by \aduca\ provide the tightest Lipschitz estimates. 

Finally, while \aduca\ is conceptually inspired by two prior lines of work---a cyclic coordinate algorithm \coder~\cite{SongDiakonikolas2023CODER} and the adaptive full-operator algorithm  \graal~\cite{malitsky2020golden,Alacaoglu2023BeyondGR}---the resulting algorithm is different and does not lead to either of these two algorithms as a special case. Specifically, when $m = 1$, ignoring the cyclic updates and comparing to \graal, \aduca\ uses operator information delayed by one iteration. On the other hand, compared to \coder, \aduca\ includes an additional iterate extrapolation step~\cref{eq:vk} and does not rely on dual averaging.
Regarding the locally adaptive step sizes, compared to \graal, our scheme introduces an additional upper bound with respect to $\hat{L}_k$ in every epoch in the stepsize selection condition to account for cyclic updates, distinguishing it from~\cite{malitsky2020golden}.
In the strongly convex case, our scheme~\cref{eq:complete-step-size-conditions} and the corresponding analysis differ from~\cite{malitsky2020golden}, and directly lead to near-linear convergence without requiring a lower bound on the step sizes as in~\cite{malitsky2020golden}. Finally, to control errors from the cyclic updates, the possible geometric growth of stepsizes here is slightly smaller than the ones  in~\cite{malitsky2019adaptive,Alacaoglu2023BeyondGR}, with a smaller $\rho_0$ imposed on it.
 \section{Convergence Analysis}\label{sec:analysis}

In this section, {we} carry out the convergence analysis of \aduca\ by bounding the individual merit functions $h_k(\vu),$ as mentioned before. While bounding similar ``gap-like'' functions is standard (and one could argue natural, since it directly translates into the target error guarantee), the specific argument carried out in bounding $h_k(\vu)$ is, of course, specific to this work. The first step in the analysis is to decompose the merit function, as stated below:
\begin{equation}\label{eq:merit-decomposition}
    h_k(\vu) = a_k\innp{\mF(\vu_k) - \mFb_k, \vu_k - \vu} + a_k\big(\innp{\mFb_k, \vu_k - \vu} + g(\vu_k) - g(\vu)\big).
\end{equation}
The decomposition is based on replacing $\mF(\vu_k)$ with the extrapolated term {$\mFb_k$} and controlling the introduced ``error'' $h_k^{(1)}(\vu):= a_k\innp{\mF(\vu_k) - \mFb_k, \vu_k - \vu}.$ In particular, similar to prior work utilizing extrapolated operators \cite{kotsalis2022simple,SongDiakonikolas2023CODER,diakonikolas2025block}, we split $h_k^{(1)}(\vu)$ into telescoping inner product terms and an error term that we later cancel out (Lemma~\ref{lemma:op-correction}). The remaining term 
$h_k^{(2)}(\vu):= a_k\big(\innp{\mFb_k, \vu_k - \vu} + g(\vu_k) - g(\vu)\big)$
on the right-hand side of \eqref{eq:merit-decomposition} is controlled by the algorithm update and we show {it} contains negative quadratic distance terms  scaling with $\|\vu_{k+1} - \vu_k\|_{\mLambda}^2$, which are used in controlling any introduced ``error'' terms (Lemma~\ref{lemma:update-merit-error}). 

Before delving into the technical details, we describe the rationale behind the different algorithmic choices in our algorithm. First, the operator extrapolation $\mFb_k$ is used to ensure that the updates to $\vu_{k+1}$ indeed correspond to cyclic block coordinate updates. Here, the delayed nature of $\mFb_k$ is essential to ensuring that we can select step size parameters $a_k$ \emph{adaptively} to local geometry; we comment on this more as  we carry out the analysis below. The update $\vu_{k+1}$ resembles the \graal\ update from \cite{malitsky2020golden}, but with the crucial difference that the full operator $\mF(\vu_k)$ is replaced by the extrapolated one {with cyclic block updates}---$\mFb_k$. While this may appear as a small difference, it nevertheless has a substantial impact on the analysis, leading to more challenges in setting the conditions on the step size update and requiring control of squared distance terms $\|\vu_{k+1} - \vu_k\|_{\mLambda}^2$  not only over two consecutive iterations (as commonly happens for similar algorithms \cite{malitsky2020golden,SongDiakonikolas2023CODER,diakonikolas2025block,kotsalis2022simple}), but over \emph{four} consecutive iterations. 

We begin now with a simple lemma that bounds $h_k^{(1)}(\vu).$

\begin{lemma}\label{lemma:op-correction}
    Let $h_k^{(1)}(\vu):= a_k\innp{\mF(\vu_k) - \mFb_k, \vu_k - \vu},$ where, as in \eqref{eq:barF}, $\mFb_k = \mFt_k + \frac{a_{k-1}\omega_{k-1}}{a_k}(\mF(\vu_{k-1}) - \mFt_{k-1})$. Then:
\begin{equation}\notag
        \begin{aligned}
            h_k^{(1)}(\vu) =\;& a_k\innp{\mF(\vu_k) - \mFt_k, \vu_{k+1} - \vu} - a_{k-1} \omega_{k-1}\innp{\mF(\vu_{k-1}) - \mFt_{k-1}, \vu_{k} - \vu}\\
            &+ a_k\innp{\mF(\vu_k) - \mFt_k, \vu_{k} - \vu_{k+1}}.
        \end{aligned}
    \end{equation}
\end{lemma}
\begin{proof}
    Plugging in the definition of $\mFb_k,$ we get  $h_k^{(1)}(\vu) = a_k\innp{\mF(\vu_k) - \mFt_k, \vu_{k} - \vu} - a_{k-1} \omega_{k-1}\innp{\mF(\vu_{k-1}) - \mFt_{k-1}, \vu_{k} - \vu}$. To complete the proof, it remains to add and subtract $a_k\innp{\mF(\vu_k) - \mFt_k, \vu_{k} - \vu_{k+1}},$ and simplify. 
\end{proof}
Observe here that when $h_k^{(1)}(\vu)$ is multiplied by $\theta_k,$ the first two terms on the right-hand side in Lemma \ref{lemma:op-correction} telescope, as $\theta_{k} \omega_{k-1} = \theta_{k-1},$ by the definition of $\theta_k.$ Thus, only the third inner product term $a_k\innp{\mF(\vu_k) - \mFt_k, \vu_{k} - \vu_{k+1}}$ will need to be controlled by the algorithm update. Here, it is crucial that $\mFt_k$ does not depend on the information from the current cycle, since it ensures that $\|\mF(\vu_k) - \mFt_k\|_{\mLambda^{-1}}$ is a function of local Lipschitzness between points from the previous cycle and so it is available to the algorithm at the time the step size $a_k$ is being set.

A more technical part of the proof concerns appropriately bounding $h_k^{(2)}(\vu) := a_k\big(\innp{\mFb_k, \vu_k - \vu} + g(\vu_k) - g(\vu)\big)$. Here, the inspiration comes from \graal-type algorithms \cite{Alacaoglu2023BeyondGR,malitsky2020golden}, which leverage a conveniently chosen sequence of prox-centers $\vv_k$ and the property of Euclidean norms stated in Fact~\ref{fact:quadratic-cvx-comb}. Our analysis is, however, different, as it  directly bounds the merit function and is adapted to our operator extrapolation. 

\begin{lemma}\label{lemma:update-merit-error}
    Consider the iterates $\vu_k,$ $k \geq 0,$ of Algorithm \ref{alg:ADUCA}. Then, for any $\vu \in \dom(g),$ we have that
\begin{equation}\notag
        \begin{aligned}
            h_k^{(2)}(\vu) \leq\; & 
             \frac{1}{2(1 - \beta)}\Big((1 + \beta \mu a_k)\|\vu - \vv_k\|_{\mLambda}^2 - (1 + \mu a_k)\|\vu - \vv_{k + 1}\|_{\mLambda}^2\Big)\\
            &+ \frac{\gamma \beta^2(1+\beta)}{2(1+\beta^2)}\big(\omega_{k-1}\|\vu_{k} - \vv_{k-1}\|_{\mLambda}^2 - \|\vu_{k + 1} - \vv_k\|_{\mLambda}^2\big)\\
            &+a_k\innp{\mFb_k - \mFb_{k - 1}, \vu_k - \vu_{k + 1}}\\
            &- \Big(\frac{\gamma{\omega_{k-1}}(1+\beta)}{2(1+\omega_{k-1})(1+\beta^2)} + \frac{a_k(1 + \beta \mu a_{k-1})}{2\beta a_{k - 1}}\Big)\|\vu_{k + 1} - \vu_k\|_{\mLambda}^2. 
        \end{aligned}
    \end{equation}
\end{lemma}
\begin{proof}
    Observe that we can decompose $h_k^{(2)}(\vu)$ as 
\begin{equation}\label{eq:hk-2-decomposition}
        \begin{aligned}
            h_k^{(2)}(\vu) =\; &  {a_k \big(\innp{\mFb_{k}, \vu_{k+1} - \vu} + g(\vu_{k+1}) - g(\vu)\big)} \\ 
  & + {\frac{a_{k}}{a_{k - 1}}a_{k - 1}\big(\innp{\mFb_{k - 1}, \vu_k - \vu_{k + 1}} + g(\vu_{k}) - g(\vu_{k+1})\big)}\\
  &+ a_k\innp{\mFb_k - \mFb_{k - 1}, \vu_k - \vu_{k + 1}}.
        \end{aligned}
    \end{equation}
Denote $\gT_1 : = a_k \big(\innp{\mFb_{k}, \vu_{k+1} - \vu} + g(\vu_{k+1}) - g(\vu)\big)$ and $\gT_2 := \frac{a_{k}}{a_{k - 1}}a_{k - 1}\big(\innp{\mFb_{k - 1}, \vu_k - \vu_{k + 1}} + g(\vu_{k}) - g(\vu_{k+1})\big)$, the terms in the first and the second line of the right-hand side of \eqref{eq:hk-2-decomposition}. To prove the lemma, we bound above $\gT_1 + \gT_2.$

    Observe first that by the update rule \eqref{eq:uk} and Fact \ref{fact:update-3PI}, we have
\begin{align}
  \gT_1 \leq \;& \frac{1}{2}\big(\|\vv_k - \vu\|_{\mLambda}^2  - \|\vu_{k + 1} - \vv_k\|_{\mLambda}^2 - (\mu a_{k} + 1)\|\vu_{k + 1} - \vu\|_{\mLambda}^2\big), \label{eq:gT-1-bnd} \\
  \gT_2 \leq \;& \frac{a_k}{2 a_{k - 1}}\big(\|\vu_{k + 1} - \vv_{k - 1}\|_{\mLambda}^2 - \|\vu_k - \vv_{k - 1}\|_{\mLambda}^2 - (\mu a_{k-1} + 1)\|\vu_k - \vu_{k + 1}\|_{\mLambda}^2\big). \label{eq:gT-2-bnd}
\end{align}
Further, by its definition, $\vv_k = (1-\beta)\vu_k + \beta\vv_{k-1}.$ Thus, $\|\vu - \vv_{k+1}\|_{\mLambda}^2 = \|(1-\beta)(\vu - \vu_{k+1}) + \beta(\vu - \vv_k)\|_{\mLambda}^2$. A rearrangement of the identity from Fact \ref{fact:quadratic-cvx-comb} then gives
\begin{equation}\label{eq:uk+1-u2}
        \|\vu_{k+1} - \vu\|_{\mLambda}^2 = \frac{1}{1-\beta}\|\vu - \vv_{k+1} \|_{\mLambda}^2 - \frac{\beta}{1-\beta}\|\vu - \vv_k \|_{\mLambda}^2 + \beta \|\vu_{k+1} - \vv_k \|_{\mLambda}^2. 
    \end{equation}
Similarly, since $\|\vu_{k+1} - \vv_k\|_{\mLambda}^2 = \|(1-\beta)(\vu_{k+1} - \vu_k) + \beta(\vu_{k+1} - \vv_{k-1})\|_{\mLambda}^2$, we have
\begin{equation}\label{eq:uk+1-vk-12}
        \|\vu_{k+1} - \vv_{k-1} \|_{\mLambda}^2 = \frac{1}{\beta}\|\vu_{k+1} - \vv_k \|_{\mLambda}^2 - \frac{1-\beta}{\beta}\|\vu_{k+1} - \vu_k \|_{\mLambda}^2 + (1-\beta)\|\vu_k - \vv_{k-1} \|_{\mLambda}^2.
    \end{equation}
Plugging \eqref{eq:uk+1-u2} and \eqref{eq:uk+1-vk-12} back into \eqref{eq:gT-1-bnd} and \eqref{eq:gT-2-bnd}, respectively, and summing them up, we obtain the following bound on $\gT_1 + \gT_2$:
\begin{equation}\label{eq:sum-gT1-gT2}
        \begin{aligned}
           & \gT_1 + \gT_2\\
          & \leq \frac{1}{2(1 - \beta)}\Big((1 + \beta \mu a_k)\|\vu - \vv_k\|_{\mLambda}^2 - (1 + \mu a_k)\|\vu - \vv_{k + 1}\|_{\mLambda}^2\Big) \\ 
    & + \Big(\frac{a_k}{2\beta a_{k - 1}} - \frac{1 + \beta + \beta \mu a_k}{2}\Big)\|\vu_{k + 1} - \vv_k\|_{\mLambda}^2 - \frac{a_k(1 + \beta \mu a_{k-1})}{2\beta a_{k - 1}}\|\vu_{k + 1} - \vu_k\|_{\mLambda}^2,
        \end{aligned}
    \end{equation}
where we dropped the non-positive term $- \frac{\beta a_k}{2 a_{k - 1}} \|\vu_k - \vv_{k - 1}\|_{\mLambda}^2$ from the right-hand side. It now remains to bound above $\big(\frac{a_k}{2\beta a_{k - 1}} - \frac{1 + \beta + \beta \mu a_k}{2}\big)\|\vu_{k + 1} - \vv_k\|_{\mLambda}^2.$

    Observe first that $\frac{a_k}{2\beta a_{k - 1}} \leq (1-\gamma) \frac{1 + \beta + \beta \mu a_k}{2}$ would ensure  $\frac{a_k}{2\beta a_{k - 1}} - \frac{1 + \beta + \beta \mu a_k}{2} \leq -\gamma \frac{1 + \beta + \beta \mu a_k}{2} \leq - \frac{\gamma(1+\beta)}{2}.$ A rearrangement of $\frac{a_k}{2\beta a_{k - 1}} \leq (1-\gamma) \frac{1 + \beta + \beta \mu a_k}{2}$ gives $\big(\frac{1}{1-\gamma}-\beta^2 \mu a_{k-1}\big)a_k \leq \beta(1 + \beta)a_{k-1},$ which is implied by (any of) our step size conditions (see \eqref{eq:complete-step-size-conditions}), and thus:
\begin{equation}\label{eq:uk+1-vk-bnd-1}
        \Big(\frac{a_k}{2\beta a_{k - 1}} - \frac{1 + \beta + \beta \mu a_k}{2}\Big)\|\vu_{k + 1} - \vv_k\|_{\mLambda}^2 \leq - \frac{\gamma(1+\beta)}{2}\|\vu_{k + 1} - \vv_k\|_{\mLambda}^2. 
    \end{equation}
Now observe that we can write $\vu_{k + 1} - \vv_k = \vu_{k + 1} - \vu_k + \vu_k - \vv_k = \vu_{k + 1} - \vu_k + \beta(\vu_k - \vv_{k-1})$. Thus, applying Fact \ref{fact:reverse-of-young} with $\alpha = {\omega_{k-1}},$ we have
\begin{equation}\notag
        \|\vu_{k + 1} - \vv_k\|_{\mLambda}^2 \geq \frac{\omega_{k-1}}{1+\omega_{k-1}}\|\vu_{k + 1} - \vu_k\|_{\mLambda}^2 - \omega_{k-1}\beta^2 \|\vu_k - \vv_{k-1}\|_{\mLambda}^2. 
    \end{equation}
Taking a convex combination of the above inequality and the trivial inequality $\|\vu_{k + 1} - \vv_k\|_{\mLambda}^2 \geq \|\vu_{k + 1} - \vv_k\|_{\mLambda}^2$, with weights $\frac{1}{1+\beta^2}$ and $\frac{\beta^2}{1 + \beta^2},$ we get
\begin{equation}\label{eq:uk+1-vk-bnd-2} 
        \begin{aligned}
            \|\vu_{k + 1} - \vv_k\|_{\mLambda}^2 \geq \; &\frac{{\omega_{k-1}}}{(1+\beta^2)(1+\omega_{k-1})}\|\vu_{k + 1} - \vu_k\|_{\mLambda}^2 \\
            &+\frac{\beta^2}{1+\beta^2}\big(\|\vu_{k + 1} - \vv_k\|_{\mLambda}^2 - \omega_{k-1}\|\vu_{k} - \vv_{k-1}\|_{\mLambda}^2\big). 
        \end{aligned}
    \end{equation}
To complete the proof, it now remains to combine \eqref{eq:hk-2-decomposition}, \eqref{eq:sum-gT1-gT2}, \eqref{eq:uk+1-vk-bnd-1}, and \eqref{eq:uk+1-vk-bnd-2}, and group the like terms. 
\end{proof}

The bound on $h_k^{(2)}(\vu)$ in Lemma \ref{lemma:update-merit-error} may be difficult to parse, so we explain here what it states. The terms in the first two lines of the bound telescope when multiplied by $\theta_k,$ which is ensured by $\theta_k = \theta_{k-1}/\omega_{k-1}$ and our conditions on step sizes $a_k$. The remaining two lines contain an ``error'' that we show how to control in the sequel.

What we have achieved so far is to show that we can bound the merit function $h_k(\vu)$ by terms that telescope when multiplied by $\theta_k$ and terms that all can be expressed as functions of increments $\vu_{k+1} - \vu_k.$ In particular, Lemmas \ref{lemma:op-correction} and \ref{lemma:update-merit-error} imply
\begin{equation}\label{eq:h_k-high-level-decomposition}
    \theta_k h_k(\vu) \leq [\text{telescoping terms}]_k + \theta_k \gE_k, 
\end{equation}
where
\begin{equation}\label{eq:error_k-def}
\begin{aligned}
    \gE_k :=\; & a_k\innp{\mF(\vu_k) - \mFt_k, \vu_{k} - \vu_{k+1}} + a_k\innp{\mFb_k - \mFb_{k - 1}, \vu_k - \vu_{k + 1}}\\
            &- \Big(\frac{\gamma{\omega_{k-1}}(1+\beta)}{2(1+\omega_{k-1})(1+\beta^2)} + \frac{a_k(1 + \beta \mu a_{k-1})}{2\beta a_{k - 1}}\Big)\|\vu_{k + 1} - \vu_k\|_{\mLambda}^2. 
\end{aligned}
\end{equation}
The key step in the analysis is now to show that $\gE_k$ can be effectively controlled, using local Lipschitz assumptions on $\mF$ and by appropriately selecting the (adaptive) step size $a_k.$ This is done in the following lemma.

\begin{lemma}\label{lemma:error-k}
     Consider the iterates $\vu_k,$ $k \geq 0,$ of Algorithm \ref{alg:ADUCA} and $\gE_k$ defined by \eqref{eq:error_k-def}. Then, under Assumption \ref{assp:localLipschitz}, for all $k \geq0,$
\begin{equation}\notag
     \begin{aligned}
         \theta_k \gE_k \leq\; & - \frac{c_k}{2}\|\vu_{k+1} - \vu_{k}\|_{\mLambda}^2\\
         &+ \frac{c_{k-1}}{8}\|\vu_{k} - \vu_{k-1}\|_{\mLambda}^2 + \frac{c_{k-2}}{4}\|\vu_{k-1} - \vu_{k-2}\|_{\mLambda}^2 + \frac{c_{k-3}}{8}\|\vu_{k-2} - \vu_{k-3}\|_{\mLambda}^2,
     \end{aligned}
     \end{equation}
where $c_k := \frac{\theta_k a_k(\beta \mu a_{k-1} + 1)}{\beta a_{k - 1}}$ for $k \geq 0$ and $c_i = 0$ for $i < 0.$
\end{lemma}
\begin{proof}
By the definition of $\mFb_k$ in~\cref{eq:barF}, the first two terms of $\gE_k$, $a_k\innp{\mF(\vu_k) - \mFt_{k}, \vu_k - \vu_{k + 1}} + a_k\innp{\mFb_k - \mFb_{k - 1}, \vu_k - \vu_{k + 1}}$, can be decomposed as follows: 
\begin{equation*}
\begin{aligned}
  \;& a_k\innp{\mF(\vu_k) - \mFt_{k}, \vu_k - \vu_{k + 1}} + a_k\innp{\mFb_k - \mFb_{k - 1}, \vu_k - \vu_{k + 1}} \\
  = \;& a_k\innp{\mF(\vu_k) - \mF(\vu_{k - 1}), \vu_k - \vu_{k + 1}} + a_{k - 1} \omega_{k - 1} \innp{\mF(\vu_{k - 1}) - \mFt_{k - 1}, \vu_{k} - \vu_{k + 1}} \\
  & + a_k\innp{\mF(\vu_{k - 1}) - \mFt_{k - 1}, \vu_{k} - \vu_{k + 1}} + \frac{a_k a_{k - 2} \omega_{k - 2}}{a_{k- 1}}\innp{\mF(\vu_{k - 2}) - \mFt_{k - 2}, \vu_{k + 1} - \vu_k}.
\end{aligned}
\end{equation*}
Recall here that $\vu_i = \vu_0$ for $i < 0$ and $\mFt_i = \mF(\vu_0)$ for $i \leq 0$. 
To bound the non-zero terms on the right-hand side, we apply Cauchy-Schwarz inequality followed by Young's inequality, which leads to the following statement: for any $\vu, \vv, \vu', \vv' \in \sR^d$,~$\delta > 0$, \begin{align*}\label{eq:ineq-chain}
    \innp{\vu' - \vv', \vu - \vv} \leq \big\|\mLambda^{-1/2}(\vu' - \vv')\big\| \big\|\mLambda^{1/2}(\vu - \vv)\big\| \leq \frac{\delta}{2}\|\vu' - \vv'\|^{2}_{\mLambda^{-1}} + \frac{1}{2\delta}\|\vu - \vv\|^{2}_{\mLambda}.
\end{align*} 
Applying this inequality to each term with parameters $\delta_1, \delta_2, \delta_3, \delta_4 > 0$, and recalling the definitions $L_k$ and $\hat L_k,$ we obtain for $k \geq 3:$
\begin{equation}\notag \begin{aligned}
    \;& a_k\innp{\mF(\vu_k) - \mFt_{k}, \vu_k - \vu_{k + 1}} + a_k\innp{\mFb_k - \mFb_{k - 1}, \vu_k - \vu_{k + 1}} \\
    \leq \;& \frac{1}{2}\Big(\frac{a_k}{\delta_1} + \frac{a_k}{\delta_2} + \frac{a_{k - 1} \omega_{k - 1}}{\delta_3} + \frac{a_k a_{k - 2} \omega_{k - 2}}{a_{k - 1}\delta_4}\Big)\|\vu_{k + 1} - \vu_k\|_{\mLambda}^2 \\
    & + \frac{L_k^2 a_k \delta_1}{2}\|\vu_k - \vu_{k - 1}\|_{\mLambda}^2 + \frac{\hat L_{k - 1}^2 (a_k\delta_2 + a_{k - 1}\omega_{k - 1}\delta_3)}{2}\|\vu_{k - 1} - \vu_{k - 2}\|_{\mLambda}^2 \\
    & + \frac{\hat L_{k - 2}^2 a_k a_{k - 2} \omega_{k - 2}\delta_4}{2a_{k - 1}}\|\vu_{k - 2} - \vu_{k - 3}\|_{\mLambda}^2. 
\end{aligned}
\end{equation}    
For more general $k,$ a similar expression is obtained, where we note that the terms involving $\delta_1$ only appear for $k \geq 1,$ the terms involving $\delta_2$ and $\delta_3$ appear for $k \geq 2$, and the terms involving $\delta_4$ appear for $k \geq 3$ (otherwise these terms should be treated as zero).  
Thus, we have (with the same note about $k$): 
\begin{equation}\label{eq:Ek-with-deltas}
    \begin{aligned}
        \theta_k\gE_k \leq \;& \frac{\theta_k}{2}\Big(\frac{a_k}{\delta_1} + \frac{a_k}{\delta_2} + \frac{a_{k - 1} \omega_{k - 1}}{\delta_3} + \frac{a_k a_{k - 2} \omega_{k - 2}}{a_{k - 1}\delta_4}\\
        &\quad- \frac{\gamma{\omega_{k-1}}(1+\beta)}{(1+\omega_{k-1})(1+\beta^2)} - \frac{a_k(1 + \beta \mu a_{k-1})}{\beta a_{k - 1}}\Big)\|\vu_{k + 1} - \vu_k\|_{\mLambda}^2 \\
    & + \frac{\theta_k L_k^2 a_k \delta_1}{2}\|\vu_k - \vu_{k - 1}\|_{\mLambda}^2 + \frac{\theta_k\hat L_{k - 1}^2 (a_k\delta_2 + a_{k - 1}\omega_{k - 1}\delta_3)}{2}\|\vu_{k - 1} - \vu_{k - 2}\|_{\mLambda}^2 \\
    & + \frac{\theta_k\hat L_{k - 2}^2 a_k a_{k - 2} \omega_{k - 2}\delta_4}{2a_{k - 1}}\|\vu_{k - 2} - \vu_{k - 3}\|_{\mLambda}^2.
    \end{aligned}
\end{equation}
For the terms on the right-hand side of \eqref{eq:Ek-with-deltas} to telescope when summed over $k \geq 0$, we need that there exists a sequence of positive numbers $c_k,$ $k \geq 0$, such that the coefficient multiplying $\|\vu_{k+1} - \vu_k\|_{\mLambda}^2$ in \eqref{eq:Ek-with-deltas} is at most $-c_k/2,$ while the coefficients multiplying $\|\vu_{k} - \vu_{k-1}\|_{\mLambda}^2$, $\|\vu_{k-1} - \vu_{k-2}\|_{\mLambda}^2$, {$\|\vu_{k-2} - \vu_{k-3}\|_{\mLambda}^2$} are, respectively, at most $\alpha_1 c_{k-1}/2,$ $\alpha_2 c_{k-2}/2,$ $(1-\alpha_1 - \alpha_2)c_{k-3}/2$, for some positive $\alpha_1, \alpha_2$ such that $\alpha_1 + \alpha_2 < 1.$ Here we recall that, by the above discussion, any terms that would nominally correspond to $c_i$ for $i < 0$ would be zero and thus can safely be ignored.  

In particular, let us choose $\delta_1, \delta_2, \delta_3, \delta_4$ so the corresponding coefficients multiplying $\|\vu_{k} - \vu_{k-1}\|_{\mLambda}^2$, $\|\vu_{k-1} - \vu_{k-2}\|_{\mLambda}^2$, $\|\vu_{k} - \vu_{k-1}\|_{\mLambda}^2$ are, respectively, $c_{k-1}/8,$ $c_{k-2}/4,$ and $c_{k-3}/8.$ Solving the equations resulting from such conditions, we get the following expressions for $\delta_1 (k \geq 1), \delta_2 (k \geq 2), \delta_3 (k \geq 2), \delta_4 (k \geq 3)$:
\begin{equation}\label{eq:deltas-def}
    \begin{aligned}
        \delta_1 &=\frac{c_{k-1}}{4\theta_k L_k^2 a_k}, &&\delta_2 = \frac{c_{k-2}}{4\theta_k\hat{L}_{k-1}^2 a_k},\\
        \delta_3 &= \frac{c_{k-2}}{4\theta_k \hat{L}_{k-1}^2 a_{k-1}\omega_{k-1}},  &&\delta_4 = \frac{c_{k-3}a_{k-1}}{4\theta_k\hat{L}_{k-2}^2a_k a_{k-2}\omega_{k-2}}.
    \end{aligned}
\end{equation}
Thus, interpreting any terms containing $c_i$ with $i < 0$ as zero, and plugging the choices of \eqref{eq:deltas-def} into \eqref{eq:Ek-with-deltas}, the goal now becomes to ensure that the coefficient multiplying $\|\vu_{k+1} -\vu_k\|_{\mLambda}^2$ is at most {$-c_k/2.$} Recalling that $\theta_k \omega_{k-1} = \theta_{k-1},$ we want to ensure: 
\begin{equation}\notag
    \begin{aligned}
       & \frac{(2\theta_k a_k L_k)^2}{c_{k-1}} + \frac{(2\theta_k a_k \hat{L}_{k-1})^2}{c_{k-2}} + \frac{(2\theta_{k-1}a_{k-1} \hat{L}_{k-1})^2}{c_{k-2}} + \frac{1}{c_{k-3}}\Big(\frac{2\hat{L}_{k-2}\theta_k a_k \theta_{k-2}a_{k-2}}{\theta_{k-1}a_{k-1}}\Big)^2\\
       & - \frac{\theta_k\gamma{\omega_{k-1}}(1+\beta)}{(1+\omega_{k-1})(1+\beta^2)} - \frac{\theta_k a_k(1 + \beta \mu a_{k-1})}{\beta a_{k - 1}} \leq -c_k.
    \end{aligned}
\end{equation}
To ensure this inequality is satisfied, we define $c_k := \frac{\theta_k a_k(\beta \mu a_{k-1} + 1)}{\beta a_{k - 1}}$, and use the remaining negative term $-\frac{\theta_k\gamma{\omega_{k-1}}(1+\beta)}{(1+\omega_{k-1})(1+\beta^2)}$ to cancel out all of the non-negative terms in the above inequality. Fix $\tau \in (0, 1).$  
Defining $\eta := \sqrt{\frac{\gamma(1+\beta)}{1 + \beta^2}}$ to simplify the notation, we reach the following conditions on the step size:

Setting $\frac{(2\theta_k a_k L_k)^2}{c_{k-1}} \leq \frac{\tau}{3}\eta^2 \theta_k\frac{\omega_{k-1}}{1 + \omega_{k-1}}$ 
for $k \geq 1,$  plugging in the definition of $c_{k-1}$, and solving for $a_k$ (where we recall $\theta_k \omega_{k-1} = \theta_{k-1}$), we get that the first condition is
\begin{equation}\label{eq:ak-cond-1}
    a_k \leq \frac{\eta \sqrt{\tau}}{2\sqrt{3} L_k }\frac{\omega_{k-1}}{\sqrt{1 + \omega_{k-1}}}\sqrt{\frac{a_{k-1}(\beta \mu a_{k-2} + 1)}{\beta a_{k-2}}}, \; k \geq 1.
\end{equation}
Setting $ \frac{(2\theta_k a_k \hat{L}_{k-1})^2}{c_{k-2}} \leq \frac{\tau}{3}\eta^2 \theta_k\frac{\omega_{k-1}}{1 + \omega_{k-1}}$
for $k \geq 2$, and solving for $a_k,$ the second condition is
\begin{equation}\label{eq:ak-cond-2}
    a_k \leq \frac{\eta \sqrt{\tau} }{2\sqrt{3}\hat{L}_{k-1}}\frac{\omega_{k-1}\sqrt{\omega_{k-2}}}{\sqrt{1 + \omega_{k-1}}}\sqrt{\frac{a_{k-2}(\beta \mu a_{k-3} + 1)}{\beta a_{k-3}}}, \; k \geq 2.
\end{equation}
Setting $ \frac{(2\theta_{k-1}a_{k-1} \hat{L}_{k-1})^2}{c_{k-2}} \leq (1-\tau)\eta^2 \theta_k\frac{\omega_{k-1}}{1 + \omega_{k-1}}$
for $k \geq 2$, and solving for $a_{k-1}$, the third condition is
\begin{equation}\label{eq:ak-cond-3}
    a_{k-1} \leq \frac{\eta \sqrt{1-\tau}}{2\hat{L}_{k-1}}\frac{\sqrt{\omega_{k-2}}}{\sqrt{1 + \omega_{k-1}}}\sqrt{\frac{a_{k-2}(\beta \mu a_{k-3} + 1)}{\beta a_{k-3}}},\; k \geq 2.
\end{equation}
Finally, setting $\frac{1}{c_{k-3}}\Big(\frac{2\hat{L}_{k-2}\theta_k a_k \theta_{k-2}a_{k-2}}{\theta_{k-1}a_{k-1}}\Big)^2 \leq \frac{\tau}{3}\eta^2 \theta_k\frac{\omega_{k-1}}{1 + \omega_{k-1}}$ 
for $k \geq 3$, and solving for $a_k,$ the fourth condition is
\begin{equation}\label{eq:ak-cond-4}
    a_k \leq \frac{\eta \sqrt{\tau}}{2\sqrt{3}\hat{L}_{k-2}}\frac{\omega_{k-1}}{\sqrt{1 + \omega_{k-1}}}\sqrt{\frac{\omega_{k-3}}{\omega_{k-2}}}\frac{a_{k-1}}{a_{k-2}}\sqrt{\frac{a_{k-3}(\beta \mu a_{k-4} + 1)}{\beta a_{k-4}}},\; k \geq 3.
\end{equation}
To complete the proof, it remains to verify that the conditions \eqref{eq:ak-cond-1}--\eqref{eq:ak-cond-4} hold by our choice of the step size, which is immediate upon noting that $\omega_k \in (0, 1],$ $\forall k.$ 
\end{proof}

We are now ready to present the main convergence theorem for \aduca, derived by combining the preceding three lemmas, under a mild initialization condition that we show in the next section can be ensured by employing \emph{one} line search, at initialization, at (at most) a logarithmic cost.

\begin{theorem} \label{theorem:bound-gap}
    Under Assumptions \ref{assp:monotone-sol}, \ref{assp:gConvex}, and \ref{assp:localLipschitz}
and assuming that $a_{-1} = a_0 \leq \frac{1}{\sqrt{2}L_1}$, the output $\hat{\vu}_K$ and iterates $\{\vu_k\}_{k \geq 0}$ of Algorithm \ref{alg:ADUCA} run for $K \geq 1$ iterations satisfy for all $\vu \in \dom(g):$
\begin{align*}
        &\; A_K \Gap(\hat{\vu}_K;\vu) + \frac{\theta_K}{4(1- \beta)}\|\vu - \vv_{K+1}\|_\mLambda^2\\
        <\; & (1 + \mu/L_1 )\|\vu - \vu_0 \|_{\mLambda}^2 + \big({\eta^2 \beta^2} + 1 + \mu/L_1 \big)\|\vu^* - \vu_0\|_{\mLambda}^2,
    \end{align*}
where $A_K = \sum_{k=1}^K \theta_k a_k,$ {$\hat \vu_K = \frac{\sum_{k = 1}^{K} \theta_k a_k \vu_{k}}{A_K},$ and $L_1 = \frac{\| F(\vu_1) - F(\vu_{0}) \|_{\mLambda^{-1}}}{\|\vu_1 - \vu_{0}\|_{\mLambda}}.$} As a consequence, $\Gap(\hat{\vu}_K;\vu) = \cO\big(\frac{\|\vu - \vu_0\|_\mLambda^2 + \|\vu^* - \vu_0\|_\mLambda^2}{A_K}\big)$ and $\|\vu^* - \vv_{K+1}\|_\mLambda^2 = \cO\big(\frac{\|\vu^* - \vu_0\|_\mLambda^2}{\theta_K}\big).$
\end{theorem}
\begin{proof}
    Recall that $h_k(\vu) = h_k^{(1)}(\vu) + h_k^{(2)}(\vu)$. Using bounds from Lemmas \ref{lemma:op-correction}, \ref{lemma:update-merit-error}, and \ref{lemma:error-k}, summing up $\theta_k h_k(\vu)$ and noticing that the sum of the total bound is telescoping (where we recall that $\eta^2 = \frac{\gamma(1+\beta)}{1+\beta^2}$ and $\mFt_0 = \mF(\vu_0)$), we get
\begin{equation}\label{eq:summed-gaps}
        \begin{aligned}
           &\;  \sum_{k=1}^K \theta_k h_k(\vu)\\
           \leq \; &\theta_K a_K\innp{\mF(\vu_K) - \mFt_K, \vu_{K+1} - \vu} - \frac{3c_{K-1}}{8}\|\vu_K - \vu_{K-1}\|_\mLambda^2\\
            &+ \frac{1}{2(1 - \beta)}\Big(\theta_1(1 + \beta \mu a_1)\|\vu - \vv_1\|_{\mLambda}^2 - \theta_K(1 + \mu a_K)\|\vu - \vv_{K + 1}\|_{\mLambda}^2\Big)\\
            &+ \frac{\eta^2 \beta^2}{2}\big(\theta_{0}\|\vu_{1} - \vv_{0}\|_{\mLambda}^2 - \theta_K\|\vu_{K + 1} - \vv_K\|_{\mLambda}^2\big).
        \end{aligned}
    \end{equation}
Similarly as in the proof of Lemma \ref{lemma:error-k}, using generalized Cauchy-Schwarz inequality, Young's inequality, and the definition of $\hat{L}_k,$ we have that for any $\delta_1, \delta_2 > 0,$
\begin{align*}
        &\; \theta_K a_K\innp{\mF(\vu_K) - \mFt_K, \vu_{K+1} - \vu}\\
        =\; & \theta_K a_K\innp{\mF(\vu_K) - \mFt_K, \vu_{K+1} - \vv_{K+1}} + \theta_K a_K\innp{\mF(\vu_K) - \mFt_K, \vv_{K+1} - \vu}\\
        \leq \; &\frac{(\theta_K a_K \hat{L}_K)^2(\delta_1 + \delta_2)}{2}\|\vu_{K} - \vu_{K-1}\|_{\mLambda}^2 + \frac{1}{2\delta_1}\|\vu - \vv_{K+1}\|_{\mLambda}^2 + \frac{1}{2\delta_2}\|\vv_{K+1} - \vu_{K+1}\|_{\mLambda}^2\\
        =\; & \frac{(\theta_K a_K \hat{L}_K)^2(\delta_1 + \delta_2)}{2}\|\vu_{K} - \vu_{K-1}\|_{\mLambda}^2 + \frac{1}{2\delta_1}\|\vu - \vv_{K+1}\|_{\mLambda}^2 + \frac{\beta^2}{2\delta_2}\|\vv_{K} - \vu_{K+1}\|_{\mLambda}^2, 
    \end{align*}
where in the last line we recalled that $\vv_{K+1} = (1 - \beta)\vu_{K+1} + \beta\vv_K,$ by its definition. Comparing the last inequality to \eqref{eq:summed-gaps}, we get that by setting $\delta_1 = \frac{2(1 - \beta)}{\theta_K(1+\mu a_K)},$ $\delta_2 = \frac{1}{\eta^2 \theta_K}$ and verifying that under this choice our step size conditions \eqref{eq:complete-step-size-conditions} together with the hyperparameter choice (\cref{eq:ak-cond-3}, \eqref{eq:ak-cond-4}) ensure $\frac{(\theta_K a_K \hat{L}_K)^2(\delta_1 + \delta_2)}{2} \leq \frac{3 c_{K-1}}{8},$ we get that \eqref{eq:summed-gaps} simplifies to
\begin{equation}\label{eq:sum-gap-simpler}
        \sum_{k=1}^K \theta_k h_k(\vu) \leq \frac{1 + \mu a_0 }{2(1-\beta)}\|\vu - \vv_1 \|_{\mLambda}^2 + \frac{\eta^2 \beta^2}{2} \|\vu_1 - \vu_0\|_{\mLambda}^2 - \frac{\theta_K(1 + \mu a_K)}{4(1- \beta)}\|\vu - \vv_{K+1}\|_\mLambda^2,
    \end{equation}
where we used that $\theta_0 = 1$, $\theta_1 = \frac{1 + \mu a_0}{1 + \rho \beta \mu a_0} \le \frac{1 + \mu a_0}{1 + \beta \mu a_1}$ (since $a_1 \leq \rho a_0,$ by the step size conditions).  
To simplify \eqref{eq:sum-gap-simpler} further, recall that $\vv_1 = (1-\beta)\vu_1 + \beta\vv_0$, $\vv_0 = \vu_0$, and {$a_0 \leq 1/L_1$.} Thus, recalling {that $\frac{\sqrt{5} - 1}{2} \leq \beta \leq 1,$} using \Cref{fact:quadratic-cvx-comb}, and dropping the negative term from the right-hand side 
    \begin{equation}
        \frac{1 + \mu a_0 }{2(1-\beta)}\|\vu - \vv_1 \|_{\mLambda}^2 < (1 + \mu/L_1)\|\vu - \vu_0\|_{\mLambda}^2 + \frac{1 + \mu/L_1}{2}\|\vu_1 - \vv_0\|_\mLambda^2
    \end{equation}
    and we get
\begin{align*}
       &\; \sum_{k=1}^K \theta_k h_k(\vu) + \frac{\theta_K(1 + \mu a_K)}{4(1- \beta)}\|\vu - \vv_{K+1}\|_\mLambda^2\\
        <\; & (1 + \mu/L_1 )\|\vu - \vu_0 \|_{\mLambda}^2 + \frac{\eta^2 \beta^2 + 1 + \mu/L_1}{2}\|\vu_1 - \vu_0\|_{\mLambda}^2.
    \end{align*}
To complete the proof, it remains to note that by Lemma \ref{lemma:gap}, $\sum_{k=1}^K \theta_k h_k(\vu) \geq A_K \Gap(\hat{\vu}_K; \vu)$ and to argue that $\|\vu_1 - \vu_0\|_{\mLambda}^2 \leq 2 \|\vu^* - \vu_0\|_\mLambda^2,$ which is done in what follows. 
By the definition of $\vu_1$ and Fact \ref{fact:update-3PI}, we have
\begin{align*}
       &\;  a_0 \big(\innp{\mF(\vu_0), \vu_1 - \vu^*} + g(\vu_1) - g(\vu^*)\big)\\
       \leq \; & \frac{1}{2}\|\vu^* - \vu_0\|_{\mLambda}^2 - \frac{1}{2}\|\vu^* - \vu_1\|_{\mLambda}^2 - \frac{1}{2}\|\vu_0 - \vu_1\|_\mLambda^2. 
    \end{align*}
Rearranging the last inequality and using that  $\innp{\mF(\vu_1), \vu_1 - \vu^*} + g(\vu_1) - g(\vu^*) \geq 0$, \begin{align*}
        \frac{1}{2}\|\vu_1 - \vu_0\|_\mLambda^2 \leq a_0 \innp{\mF(\vu_0) - \mF(\vu_1), \vu^* - \vu_1} + \frac{1}{2}\|\vu^* - \vu_0\|_{\mLambda}^2 - \frac{1}{2}\|\vu^* - \vu_1\|_{\mLambda}^2.
    \end{align*}
As argued before (by Young's inequality and the definition of $L_1$), \[a_0 \innp{\mF(\vu_0) - \mF(\vu_1), \vu^* - \vu_1} \leq \frac{a_0^2 L_1^2}{2}\|\vu_1 - \vu_0\|_\mLambda^2 + \frac{1}{2}\|\vu^* - \vu_1\|_\mLambda^2,\] thus 
\begin{equation}\notag
         \frac{1}{2}\|\vu_1 - \vu_0\|_\mLambda^2 \leq \frac{a_0^2 L_1^2}{2}\|\vu_1 - \vu_0\|_\mLambda^2 + \frac{1}{2}\|\vu^* - \vu_0\|_{\mLambda}^2.
    \end{equation}
To complete the proof that $\|\vu_1 - \vu_0\|_\mLambda^2 \leq 2 \|\vu^* - \vu_0\|_\mLambda^2$, it remains to use $a_0^2 L_1^2 \leq 1/2,$ which holds by assumption. 
\end{proof}

\paragraph{Further implications under relaxed conditions} We briefly discuss here how our results apply even under more relaxed assumptions about the problem. In particular, convergence at rate $1/\theta_K$ (argued to be near-linear in the next section) can be guaranteed without assuming strong convexity for $g.$ What suffices is that $\mF$ satisfies a ``restricted strong monotonicity'' condition with a parameter $\mu>0$, meaning that there is a solution $\vu^*$ such that for all $\vu \in \dom(g)$ it holds:
\begin{equation}
    \innp{\mF(\vu) - \mF(\vu^*), \vu - \vu^*}  \geq \mu\|\vu -\vu^*\|_\mLambda^2.
\end{equation}
The reason is as follows. First, the only place in the analysis where ``monotonicity of $\mF$'' was used was to ensure that $\Gap(\vu_k; \vu) \leq h_k(\vu)$ for $\vu \geq 0.$ This condition still applies for $\vu = \vu^*;$ in fact, we have an even stronger condition that $\Gap(\vu_k; \vu) \leq h_k(\vu) - \mu \|\vu - \vu^*\|_\mLambda^2.$ Observe that if our algorithm is applied to a GMVI 
with an operator $\widehat{\mF}(\vu) = \mF(\vu) - \mu\mLambda(\vu - \vu^*)$ and regularizer $\widehat{g}(\vu) = g(\vu) + \mu \|\vu - \vu^*\|_\mLambda^2,$ it exhibits the exact same iterates as when applied to the original problem (determined by $\mF, g$). Now, for the problem with $\widehat{\mF}(\vu), \widehat{g}(\vu)$, we still maintain  $\Gap(\vu_k; \vu) \leq h_k(\vu)$, but now $\widehat{g}$ is $(2\mu)$-strongly convex. Thus (as will be argued in the sequel ), $\theta_k$ grows geometrically and $\|\vu^* - \vv_{k+1}\|_\mLambda^2 = \gO\big(\frac{\|\vu^* - \vu_0\|_\mLambda^2}{\theta_k}\big),$ so the convergence is near-linear.   

\section{Global Oracle Complexity Bound}\label{sec:rate}
In this section, we analyze the \emph{global} oracle complexity of \aduca, under the standard assumption  that $\mF$ is globally Lipschitz continuous.  
We do so under the simplified step size choice \eqref{eq:simple-step-size-conditions-unknown-mu}, since it does not require any knowledge of the problem parameters like Lipschitz constants or the strong convexity modulus $\mu.$ Observe that, for the fixed hyperparameters $\beta, \gamma, \rho$ (see \Cref{sec:algorithm} for an example of numerical values) and $\eta$ and $\tau$ defined as in \Cref{sec:algorithm}, the step size \eqref{eq:simple-step-size-conditions-unknown-mu} can simply be written as
\begin{equation}\label{eq:step-size-global}
    a_k = \min\Big\{\rho_0 a_{k-1}, s_k \sqrt{\frac{a_{k-1}}{a_{k-2}}}\Big\},
\end{equation}
where
\begin{equation}\label{eq:step-size-constants}
\begin{aligned}
     \rho_0 &:= \min\{\rho, \beta(1+\beta)(1-\gamma)\}, \quad && s_k := \min\Big\{\frac{C}{L_k},\, \frac{\Chat}{\Lhat_k}\Big\},\\
     C &:= \frac{\eta \rho}{2}\sqrt{\frac{\tau \beta}{3(1+\rho\beta)}},\quad  && \Chat:= \frac{\eta}{2}\sqrt{\frac{(1-\tau)\rho}{2}}.
\end{aligned}
\end{equation}
Here, both $C$ and $\Chat$ can be chosen as absolute (numerical) constants, independent of any problem parameters. 
Throughout this section, we denote upper bounds on the Lipschitz constants by $\hat L$ and $L$, i.e., 
$\hat L \geq \hat{L}_{k}$ and $L \geq L_k$ for all $k$. Observe that these global bounds on the Lipschitz constants imply, $\forall k \geq 1,$
\begin{equation}\label{eq:lb-on-sk}
    s_k \geq s_\infty := \min\Big\{\frac{C}{L},\, \frac{\Chat}{\Lhat}\Big\} = \Theta\Big(\min\Big\{\frac{1}{L},\, \frac{1}{\Lhat}\Big\}\Big). 
\end{equation}

\begin{algorithm}[t]
  \caption{Backtracking Initialization for \aduca} 
  \label{alg:init-robust}
  \begin{algorithmic}[1]
    \STATE \textbf{Input:} 
    $m$, $\vu_0$, $\mLambda$, $C$, $\Chat$
    \STATE \textbf{Initialize:} 
    $i = 0$, $\alpha = 2$ 
    \STATE $a_{0} = 1$
    \STATE $\vu_{1} = \argmin_{\vu}\Big\{a_{0}\innp{\mF(\vu_0), \vu} + a_{1} g(\vu) + \frac{1}{2}\|\vu - \vu_0\|^2_{\mLambda} \Big\}$
    \STATE $\mFt_{1}^{i} = \mF^{i}(\vu^{1}_{1}, \ldots , \vu^{i - 1}_{1}, \vu^{i}_{0}, \ldots , \vu^{m}_{0})$ $(i \in [m])$
    \STATE $\hat L_{1} = \frac{\|\mF(\vu_{1}) - \mFt_{1}\|_{\mLambda^{-1}}}{\|\vu_{1} - \vu_0\|_{\mLambda}}$,
           $L_{1} = \frac{\|\mF(\vu_{1}) - \mF(\vu_0)\|_{\mLambda^{-1}}}{\|\vu_{1} - \vu_0\|_{\mLambda}}$
    \STATE $a_{\mathrm{start}} = \min\Big\{\frac{C}{L_{1}}, \frac{\Chat}{\hat L_{1}}\Big\}$ 
    \REPEAT
        \STATE $a_0 = a_{\mathrm{start}}\alpha^{-i}$
        \STATE $\vu_{1} = \argmin_{\vu}\Big\{a_0\innp{ \mF(\vu_0), \vu} + a_0 g(\vu) + \frac{1}{2}\|\vu - \vu_0\|^2_{\mLambda} \Big\}$
        \STATE $L_1 = \frac{\|\mF(\vu_1) - \mF(\vu_0)\|_{\mLambda^{-1}}}{\|\vu_1 - \vu_0\|_{\mLambda}}$
        \STATE $i \gets i + 1$
    \UNTIL{$a_0 \leq \frac{1}{\sqrt{2}L_1}$}
    \RETURN $a_0$, $a_{-1} = a_0$
  \end{algorithmic}
\end{algorithm}

\subsection{Initialization}
We begin our discussion by arguing that \cref{alg:init-robust} correctly initializes \aduca\ (so \Cref{theorem:bound-gap} applies), at a logarithmic cost. This is summarized in the following lemma, for completeness. To avoid trivialities, we assume throughout that $L_1 + \Lhat_1 > 0.$ The algorithm can be modified to deal with the case that $L_1 + \Lhat_1 = 0$ (or is very close to zero) by preventing $a_0$ from being increased beyond some large threshold $M.$  The impact on the subsequent analysis is that any dependence on $s_\infty$ would need to be replaced by $M$ if it happens that $s_\infty \gg M.$ In turn, this would  lead to the oracle complexity of the algorithm scaling with $1/M,$ which can be seen as the ``easy'' case, where the algorithm converges within few iterations.  

\begin{lemma}\label{lemma:robust-initialization}
Let $a_0$ be the output of Algorithm~\ref{alg:init-robust}.
Define
\[
s_{\mathrm{pr}} \;:=\; \min\Bigl\{\frac{C}{L_{\mathrm{pr}}},\frac{\Chat}{\hat L_{\mathrm{pr}}}\Bigr\},
\]
where $L_{\mathrm{pr}}$ and $\hat L_{\mathrm{pr}}$ are $L_1$ and $\hat{L}_1$ computed before the backtracking repeat loop in \Cref{alg:init-robust} starts. 
Then, \Cref{alg:init-robust} terminates after finitely many backtracking steps and its output satisfies:
\(
\frac{s_\infty}{2} \;<\; a_0 \;\le\; \frac{1}{\sqrt{2}\,L_1}.
\)

Moreover, if $N_{\mathrm{bt}}$ denotes the number of iterations of the repeat loop, then
\[
N_{\mathrm{bt}}
\;\le\;
\Bigl\lceil \log_2\!\Bigl(\frac{s_{\mathrm{pr}}}{s_\infty}\Bigr)\Bigr\rceil
\;=\; \mathcal{O}\!\Bigl(\log\!\Bigl(\frac{s_{\mathrm{pr}}}{s_\infty}\Bigr)\Bigr).
\]
\end{lemma}
\begin{proof}
From the definition of $C$, we obtain $C \le 1/\sqrt{2}$. By the assumed global bounds, we have $L_{\mathrm{pr}}\le L$ and $\hat L_{\mathrm{pr}}\le \hat L$, hence
\(
    s_{\mathrm{pr}}= \min\bigl\{\frac{C}{L_{\mathrm{pr}}},\,\frac{\Chat}{\hat L_{\mathrm{pr}}}\bigr\} 
    \;\ge\; \min\bigl\{\frac{C}{L},\,\frac{\Chat}{\hat L}\bigr\}
    = s_\infty.
\) 
Algorithm~\ref{alg:init-robust} initializes the backtracking loop at $a_0=s_{\mathrm{pr}}$.

\paragraph{Termination and logarithmic backtracking complexity}
After $t$ halving updates, the candidate step size is 
\(
a_t \;=\; \frac{s_{\mathrm{pr}}}{2^t}.
\) 
Let
\(
t^\star \;:=\; \bigl\lceil \log_2\!\bigl(\frac{s_{\mathrm{pr}}}{s_\infty}\bigr)\bigr\rceil.
\) 
Then $a_{t^\star} \le s_\infty$. For any candidate $a\le s_\infty$, using $L_1(a)\le L$,
we get
\(
a_{t^\star}L_1(a_{t^\star}) \le a_{t^\star}L \le s_\infty L \le C \le \frac{1}{\sqrt{2}},
\) 
and therefore $a_{t^\star}\le 1/(\sqrt{2}L_1(a_{t^\star}))$. Hence, the stopping test must be satisfied no later than
iteration $t^\star$, implying finite termination and
\(
N_{\mathrm{bt}} \le t^\star = \bigl\lceil \log_2\!\bigl(\frac{s_{\mathrm{pr}}}{s_\infty}\bigr)\bigr\rceil.
\)

\paragraph{Upper bound}
Algorithm~\ref{alg:init-robust} halts only after the test
$a_0 \le 1/(\sqrt{2}L_1)$ holds at the returned pair $(a_0,\vu_1)$; therefore
$a_0 \le 1/(\sqrt{2}L_1)$.

\paragraph{Lower bound}
If the algorithm terminates without any halving update, then $a_0=s_{\mathrm{pr}}\ge s_\infty$,
hence $a_0>s_\infty/2$. 
Otherwise, let $a_{\mathrm{prev}}$ denote the step size immediately before the final halving update,
so that $a_0 = a_{\mathrm{prev}}/2$ and the stopping test failed at $a_{\mathrm{prev}}$.
By the argument above, every $a\le s_\infty$ satisfies the stopping test. Hence a failed test implies
$a_{\mathrm{prev}}>s_\infty$, and thus $a_0=a_{\mathrm{prev}}/2>s_\infty/2$, completing the proof.
\end{proof}

\subsection{Growth of sequences and global convergence bound}

To characterize the global oracle complexity of \aduca, the first step is to argue that, on average, step sizes $a_k$ are (at least) of the order of $s_\infty.$ Such a statement then immediately implies that $A_K$ is at least of the order $K s_\infty,$ leading to the bound $\Gap(\hat{\vu}_K;\vu) = \cO\big(\frac{(L + \Lhat)(\|\vu - \vu_0\|_\mLambda^2 + \|\vu^* - \vu_0\|_\mLambda^2)}{K}\big)$ implied by \Cref{theorem:bound-gap}. 

We begin with a simple claim, which is a consequence of arithmetic mean-geo\-metric mean (AM-GM) inequality, but will come useful in our subsequent analysis.

\begin{claim}\label{claim:AM-GM}
    For any $k \geq 2,$ if $a_k = s_k \sqrt{a_{k-1}/a_{k-2}},$ then $a_k + a_{k-2} \geq 2 s_k/\sqrt{\rho_0}.$ 
\end{claim}
\begin{proof}
    A rearrangement of $a_k = s_k \sqrt{a_{k-1}/a_{k-2}}$ gives $\sqrt{a_k a_{k-2}} = s_k \sqrt{a_{k-1}/a_k}.$ By the step size definition \eqref{eq:step-size-global}, $a_k \leq \rho_0 a_{k-1},$ so $\sqrt{a_k a_{k-2}} \geq s_k \sqrt{\rho_0}.$ It remains to use the AM-GM inequality, by which $(a_k + a_{k-2})/2 \geq \sqrt{a_k a_{k-2}}.$ 
\end{proof}

\begin{theorem}\label{thm:global-convergence}
    Given any $\epsilon > 0,$ the output of \aduca\ run for a sufficiently large number of iterations
    \begin{equation}\notag
    \begin{aligned}
        K = \cO\Big(&\min\Big\{\frac{(L + \Lhat)(\|\vu - \vu_0\|_\mLambda^2 + \|\vu^* - \vu_0\|_\mLambda^2)}{\epsilon},\\ &\frac{L + \Lhat}{\mu}\log\Big(\frac{\|\vu^* - \vu_0\|_\mLambda}{\epsilon}\Big)\log\Big(\frac{(L + \Lhat)(\|\vu - \vu_0\|_\mLambda^2 + \|\vu^* - \vu_0\|_\mLambda^2}{\epsilon}\Big)\Big\}\Big)
    \end{aligned}
    \end{equation}
    guarantees either  $\Gap(\hat{\vu}_K;\vu) \leq \epsilon$ for $\vu \in S$ or $\|\vu^* - \vv_{K+1}\|_\mLambda \leq \epsilon$.  
\end{theorem}
\begin{proof}
    Using \Cref{theorem:bound-gap}, we have that $\Gap(\hat{\vu}_K;\vu) = \cO\big(\frac{\|\vu - \vu_0\|_\mLambda^2 + \|\vu^* - \vu_0\|_\mLambda^2}{A_K}\big)$ and $\|\vu^* - \vv_{K+1}\|_\mLambda^2 = \cO\big(\frac{\|\vu^* - \vu_0\|_\mLambda^2}{\theta_K}\big),$ so to prove the theorem claim, it suffices to appropriately bound below $A_K$ and $\theta_K$. 
    
    We begin by arguing that $A_K$ is at least of the order $K(L + \Lhat),$ which applies for any $\mu \geq 0.$ Observe first that, by initialization (see {\cref{lemma:robust-initialization}}) and the step size rule \eqref{eq:step-size-global}, we have that $a_1 \geq a_0 \geq s_\infty/2.$ Let $k \geq 2$ be the smallest index such that $a_k < s_\infty/2$; if no such index exists, then $A_K \geq K s_\infty/2,$ and we are done, so suppose this is not the case. Observe that, since $k$ is the smallest such index, it must be $a_k = s_k \sqrt{a_{k-1}/a_{k-2}}$ and $a_{k-1} \geq s_{\infty}/2.$ {Let $J_k$ be the longest sequence of indices $k,k+1,k+2,\dots$ such that for each
$j\in J_k$, $a_j < s_\infty/2$. Write $|J_k|=t$ so that $J_k=\{k,k+1,\dots,k+t-1\}$. Without loss of generality, assume $t \ge 2$.

We first claim that if $|J_k| \ge 3$, then for every $j\in J_k$ with $j\ge k+2$, the step size update \eqref{eq:step-size-global} must take the growth branch, i.e.,
\begin{equation}\label{eq:aj-growth-branch} a_j=\rho_0 a_{j-1},\qquad \forall j\in J_k\cap\{k+2,k+3,\dots\}.
\end{equation}

Thus, the only index in $J_k$ for which the update \eqref{eq:step-size-global} may take either branch (either term in the step size definition) is
$j=k+1$. We now distinguish two cases.

\medskip\noindent\textbf{Case 1: $a_{k+1}=\rho_0 a_k$.}
Then by \eqref{eq:aj-growth-branch}, $a_{k+i}=\rho_0^i a_k$ for all $i=0,1,\dots,t-1$, hence
$a_{k+t-1}=\rho_0^{t-1}a_k<s_\infty/2$.
Using $a_k=s_k\sqrt{a_{k-1}/a_{k-2}}\ge s_\infty\sqrt{a_{k-1}/a_{k-2}}$, we get
\[
\rho_0^{t-1}\ <\ \frac{s_\infty}{2a_k}\ \le\ \frac{1}{2}\sqrt{\frac{a_{k-2}}{a_{k-1}}},
\]
and therefore $a_{k-2}\ge 4a_{k-1}\rho_0^{2t-2}\ge (s_\infty/2)\rho_0^{2t}$, where
we used $a_{k-1}\ge s_\infty/2$ and $\rho_0<2$ in the last step.

\medskip\noindent\textbf{Case 2: $a_{k+1}=s_{k+1}\sqrt{a_k/a_{k-1}}<s_\infty/2$.}
In this case $t\ge 2$, and by \eqref{eq:aj-growth-branch} we have $a_{k+i}=\rho_0^{i-1}a_{k+1}$ for
$i=1,2,\dots,t-1$, so
$a_{k+t-1}=\rho_0^{t-2}a_{k+1}<s_\infty/2$, i.e.,
\[
a_{k+1}\ <\ \frac{s_\infty}{2}\rho_0^{-(t-2)}.
\]
Moreover, since $a_k=s_k\sqrt{a_{k-1}/a_{k-2}}$ and $a_{k+1}=s_{k+1}\sqrt{a_k/a_{k-1}}$
(the second branch is taken at both $k$ and $k+1$), eliminating $a_{k-1}$ gives
\[
a_{k-2}=\frac{s_k^2s_{k+1}^2}{a_k a_{k+1}^2}.
\]
Using $s_k,s_{k+1}\ge s_\infty$ and $a_k<s_\infty/2$ yields
$a_{k-2}\ge 2s_\infty^3/a_{k+1}^2$; combining with the above bound on $a_{k+1}$ gives
$a_{k-2}\ge 8s_\infty\rho_0^{2t-4}\ge (s_\infty/2)\rho_0^{2t}$ (again using $\rho_0<2$).

Combining the two cases, we conclude that
\begin{equation}\label{eq:lb-on-ak-2}
a_{k-2}\ \ge\ (s_\infty/2)\rho_0^{2|J_k|}.
\end{equation}

Now, even if we exclude $a_{k-2}$, the first $k-1$ step sizes are sufficiently large
on average, as $A_{k-1}-a_{k-2}\ge (k-2)(s_\infty/2)$. On the other hand, from \eqref{eq:lb-on-ak-2},
\begin{equation}\label{eq:global-lb-on-ak-2}
a_{k-2}\ \ge\ (s_\infty/2)|J_k|\cdot \frac{\rho_0^{2|J_k|}}{|J_k|}
\ \ge\ (s_\infty/2)|J_k|\inf_{t\ge 1}\frac{\rho_0^{2t}}{t}.
\end{equation}
Since $\rho_0>1$, $\inf_{t\ge 1}\rho_0^{2t}/t$ is bounded below by a universal constant
$C'>0$. Thus,
$A_{k+|J_k|-1}\ge A_{k-1}\ge (k-2+C'|J_k|)(s_\infty/2)$, which is of the claimed order.
We can repeat the same argument treating the iterate $a_{k+|J_k|}$ as the initial one,
and repeat the same reasoning to conclude that $A_K$ must scale with $Ks_\infty$, which leads to the claimed order-$(1/\epsilon)$
bound on the number of iterations.
} 
Suppose now that $\mu > 0$. Partition the iterations $\{1, 2, \dots, K\}$ into sets $I_1$ and $I_2,$ where $I_1$ contains all the iteration indices $j$ such that $a_j \geq s_\infty/2$ and $I_2$ contains the rest. Recall (from \eqref{eq:theta-def}, \eqref{eq:omega-def}) that 
    \begin{equation}\notag 1/\theta_K = \prod_{k=0}^{K-1}\omega_k = \prod_{k=0}^{K-1}\Big(1 - \frac{(1 -\rho\beta)\mu a_{k}}{1 + \mu a_k}\Big).\end{equation}
    By construction, $\omega_k \in (0, 1]$ for all $k \geq 0$ (since $\rho\beta \in (0, 1)$), and thus
\begin{equation}\label{eq:conv-through-theta}
        1/\theta_K \leq \prod_{k\in I_1}\Big(1 - \frac{(1 -\rho\beta)\mu a_{k}}{1 + \mu a_k}\Big) \leq \Big(1 - \frac{(1 -\rho\beta)\mu (s_\infty/2)}{1 + \mu (s_\infty/2)}\Big)^{|I_1|}.
    \end{equation}
We can now conclude from \eqref{eq:conv-through-theta} and $\|\vu^* - \vv_{K+1}\|_\mLambda^2 = \cO\big(\frac{\|\vu^* - \vu_0\|_\mLambda^2}{\theta_K}\big)$ that whenever \[|I_1| \geq \frac{2 + \mu s_\infty}{(1 -\rho\beta)\mu (s_\infty/2)}\ln\big(\frac{\|\vu^* - \vu_0\|_{\mLambda}}{\epsilon}\big),\] it must be $\|\vu^* - \vu_0\|_{\mLambda} \leq \epsilon.$
    
    Suppose now that $|I_1| < \frac{2 + \mu s_\infty}{(1 -\rho\beta)\mu (s_\infty/2)}\ln\big(\frac{\|\vu^* - \vu_0\|_{\mLambda}}{\epsilon}\big).$ Observe that the set $I_2$ must contain at least one sequence of iteration indices $J$ of length at least $|J| \geq \lfloor \frac{K}{|I_1| + 1} \rfloor.$ By the argument above (recall the derivation of \eqref{eq:global-lb-on-ak-2}), there exists at least one $j \in \{1, \dots, K\}$ such that $a_j \geq (s_\infty/2){\rho_0}^{2|J|}.$ As a consequence,
\begin{equation}\label{eq:mu>0A_K-bnd}
        A_K \geq (s_\infty/2){\rho_0}^{2|J|} \geq (s_\infty/2){\rho_0}^{2\lfloor \frac{K}{|I_1| + 1} \rfloor}.
    \end{equation}
Thus, for a universal constant $C'' > 0$ and $K \geq C''(|I_1|+1)\log_{\rho_0}\big(\frac{\|\vu - \vu_0\|_\mLambda^2 + \|\vu^* - \vu_0\|_\mLambda^2}{s_\infty \epsilon}\big)$, it holds $\Gap(\hat{\vu}_K;\vu) \leq \epsilon.$ It remains to combine with the bound on $|I_1|$ and simplify. 
\end{proof}

 \section{Numerical Experiment}\label{sec:num-exp}
In this section, we empirically evaluate \aduca\ and compare it with several representative comparable methods\footnote{All code is available at \url{https://github.com/Yee-Millennium/ADUCA}.}: Proximal Cyclic Coordinate Method (\pccm), Cyclic cOordinate Dual avEraging with extRapolation (\coder)~\cite{SongDiakonikolas2023CODER}, 
a parameter-free variant of \coder\ based on per-epoch line-search (\texttt{CODER-LineSearch})~\cite{SongDiakonikolas2023CODER}, and 
the Golden RAtio ALgorithm (\graal)~\cite{malitsky2020golden}. \pccm\ is the straightforward cyclic proximal update for~\eqref{eq:prob}; it is a natural analogue of cyclic coordinate descent for convex minimization, but it is not guaranteed to converge for general monotone variational inequalities and thus serves as a simple reference point. \coder\ is, to our knowledge, the first cyclic coordinate method with provable convergence guarantees for~\eqref{eq:prob}, while 
\texttt{CODER-LineSearch} makes \coder\ adaptive by performing a line search in every epoch, at the cost of additional operator evaluations. \graal~is a full-operator method that is adaptive to local Lipschitz geometry and does not require prior knowledge of Lipschitz constants.

In all experiments except~\Cref{subsec:ablation} (ablation analysis), we set $\beta = 0.8$, $\gamma = 0.2$, and $\rho = 1.2$ for \aduca. We use the simplified step size rule defined in~\eqref{eq:fully-spec-step-size-exp}. For 
other
methods in comparison, we tune their parameters within the ranges suggested in the respective papers and report the best-performing configuration.

\subsection{Support Vector Machine} \label{subsec:svm}
\begin{figure}[h]
    \hspace*{\fill}
    \subfloat[\texttt{a9a}]{\includegraphics[width=0.32\textwidth]{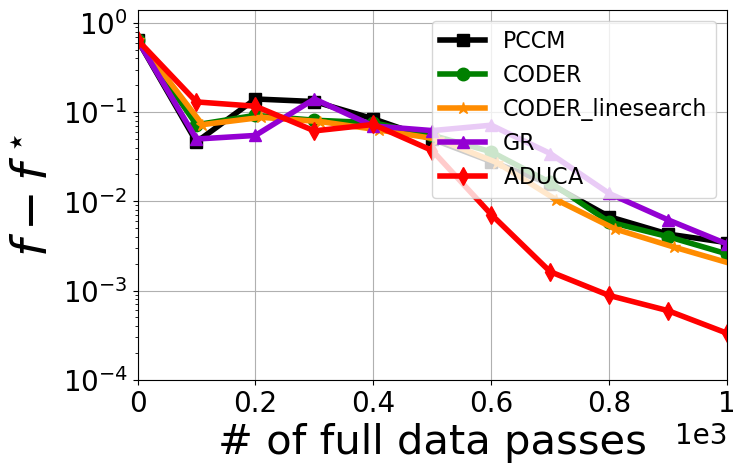}}\hfill
    \hspace*{\fill}
    \subfloat[\texttt{gisette}]{\includegraphics[width=0.32\textwidth]{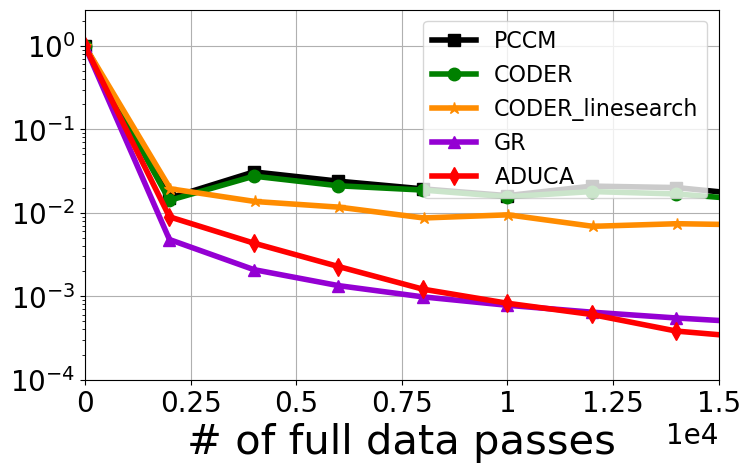}}\hfill
    \hspace*{\fill}
    \subfloat[\texttt{SUSY}]{\includegraphics[width=0.32\textwidth]{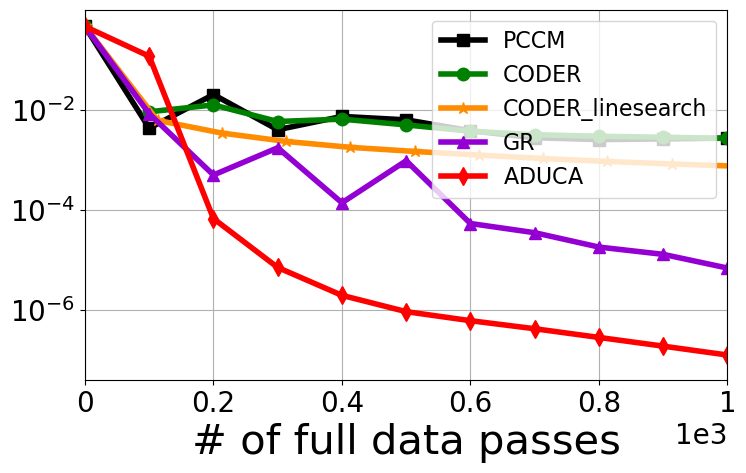}\label{fig:iteations_SUSY}}\hfill
    \hspace*{\fill}
    \subfloat[\texttt{a9a-rescaled}]{\includegraphics[width=0.32\textwidth]{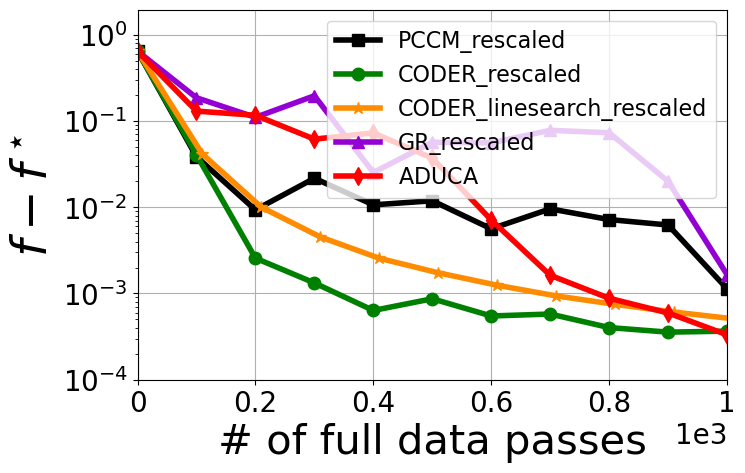}\label{fig:iteations_a9a_rescaled}}\hfill
    \hspace*{\fill}
    \subfloat[\texttt{gisette-rescaled}]{\includegraphics[width=0.32\textwidth]{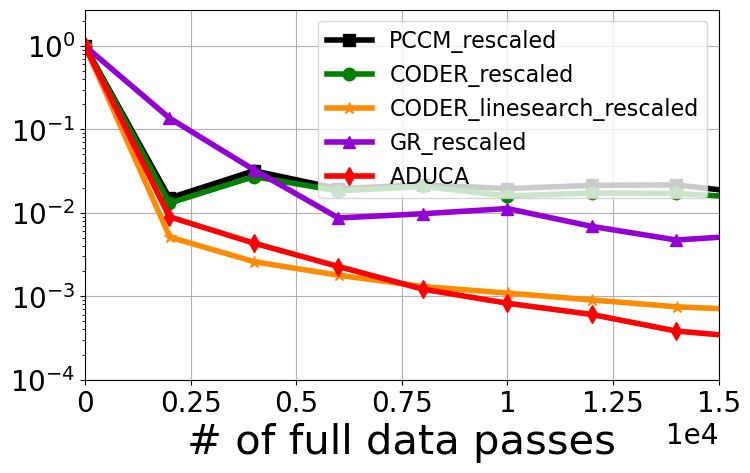}\label{fig:iteations_w8a_rescaled}}\hfill
    \hspace*{\fill}
    \subfloat[\texttt{SUSY-rescaled}]{\includegraphics[width=0.32\textwidth]{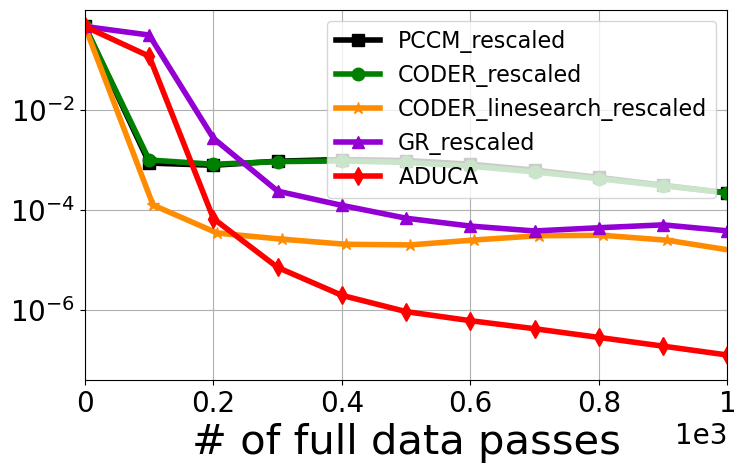}\label{fig:iteations_SUSY_rescaled}}\hfill
\caption{Primal objective suboptimality $f(\vx)-f^\star$ for~\Cref{problem:SVM} versus the number of full data passes on \texttt{a9a}, \texttt{gisette}, and \texttt{SUSY-test}. Plots (a)--(c) use the original formulation, while plots (d)--(f) use the diagonally rescaled formulation for every algorithm induced by the same diagonal matrix $\mLambda$ defined in~\eqref{eq:Lambda_svm}--\eqref{eq:Lambda_svm_entries}. In these figures, \aduca\ uses $\mLambda$ throughout (a)--(f).} \label{fig:SVM}
\end{figure}

We consider the convex elastic net-regularized support vector machine model, formulated as the following min-max  problem:
  \begin{equation} \label{problem:SVM} \tag{SVM}
      \min_{\vx \in \R^{d}} \max_{\vy \in \R^n}\frac{1}{n}\sum_{i=1}^{n}y_{i}(-1 + b_i \mA_i^{\top} \vx) + \lambda_1\| \vx \|_{1} + \frac{\lambda_2}{2}\| \vx \|^2_{2} + \sum_{j=1}^{n}\iota_{[-1,0]}(y_j),
  \end{equation}
where $\mA = [\mA_1, \ldots, \mA_n]^{\top} \in \R^{n \times d}$, $\vb \in \R^{n}$, and $\iota_{[-1,0]}(\cdot)$ is the convex indicator function of the interval $[-1,0]$. Let $\bar{\mA} = [b_1 \mA_1, \ldots, b_n \mA_n]$. This problem is an instance of~\Cref{eq:prob} with $F(\vx, \vy) = \frac{1}{n}[\bar{\mA} \vy,  \1 - \bar{\mA}^{\top}\vx] \in \R^{d+n}$, and $g(\vx, \vy) = \lambda_1 \| \vx \|_{1} + \frac{\lambda_2}{2}\| \vx \|^2_{2} + \sum_{j=1}^{n}\iota_{[-1,0]}(y_j)$. We set $\lambda_1 = 10^{-4}$ and $\lambda_2 = 10^{-4}$. Let $\vu=(\vx,\vy)\in\R^{d}\times\R^{n}$. We endow $\R^{d+n}$ with a diagonal rescaling matrix
\begin{equation}\label{eq:Lambda_svm}
\mLambda \;:=\;
\begin{bmatrix}
\mLambda_\vx & 0\\
0 & \mLambda_\vy
\end{bmatrix}
\in\R^{(d+n)\times(d+n)},\;
\mLambda_{\vx}=\text{Diag}(\lambda^\vx_1,\ldots,\lambda^\vx_d),\;
\mLambda_{\vy}=\text{Diag}(\lambda^\vy_1,\ldots,\lambda^\vy_n),
\end{equation}
with diagonal entries defined by the reciprocal $\ell_2$-norms of the rows/columns of $\bar A$:
\begin{equation}\label{eq:Lambda_svm_entries}
\begin{aligned}
    \lambda^\vx_j :=
    \begin{cases}
    \|\bar \mA_{j,:}\|_2^{-1}, & \text{if }\|\bar \mA_{j,:}\|_2>0,\\[2pt]
    1, & \text{if }\|\bar \mA_{j,:}\|_2=0,
    \end{cases}
    \quad j=1,\ldots,d, \\
    \lambda^\vy_i :=
    \begin{cases}
    \|\bar \mA_{:,i}\|_2^{-1}, & \text{if }\|\bar \mA_{:,i}\|_2>0,\\[2pt]
    1, & \text{if }\|\bar \mA_{:,i}\|_2=0,
    \end{cases}
    \quad i=1,\ldots,n.
\end{aligned}
\end{equation}
We evaluate all methods on the LibSVM~\cite{chang2011libsvm} datasets \texttt{a9a}~$(d=123, n=32561)$, \texttt{gisette}~$(d=5000, n=6000)$, and \texttt{SUSY}~$(d=18, n=5000000)$. 

To quantify progress, we report the primal gap $f(\vx)-f^\star$, where $f$ is the primal objective obtained by maximizing over $\vy$ in~\eqref{problem:SVM}, and $f^\star$ is a reference value taken as the best primal objective value attained across all methods after a long run.

\Cref{fig:SVM} plots $f(\vx)-f^\star$ versus the number of full data passes. Across the three datasets, \aduca\ consistently achieves either the fastest or competitive-with-the-fastest reduction in suboptimality. This advantage can be attributed to its cyclic updates combined with locally estimated adaptive step sizes that leverage both full-operator and block-coordinate operator information, enabling \aduca\ to exploit block structure while avoiding the per-epoch line-search overhead of \texttt{CODER-LineSearch}. 

The first row of plots in \Cref{fig:SVM} compares \aduca\ to other methods implemented using standard $\ell_2$ geometry; the second row of plots is for the implementation of all methods using the Euclidean norms rescaled by $\mLambda,$ $\|\cdot\|_{\mLambda},$ to evaluate whether the performance improvements of \aduca\ are purely due to the row/column rescaling by $\mLambda$. We see that this is not the case, since a similar rescaling does not consistently benefit other algorithms and can, in fact, harm their performance  (consider, for instance, the performance of \graal\ on the gisette dataset, in the middle column of \Cref{fig:SVM}). 

\subsection{Ablation analysis for the hyperparameter \texorpdfstring{$\mu$}{mu}} \label{subsec:ablation}
\begin{figure}[t]
    \hspace*{\fill}
    \subfloat[\texttt{a9a,}\\ $\beta = 0.7,\, \rho = 1.3,\, \gamma = 0.05.$]{\includegraphics[width=0.32\textwidth]{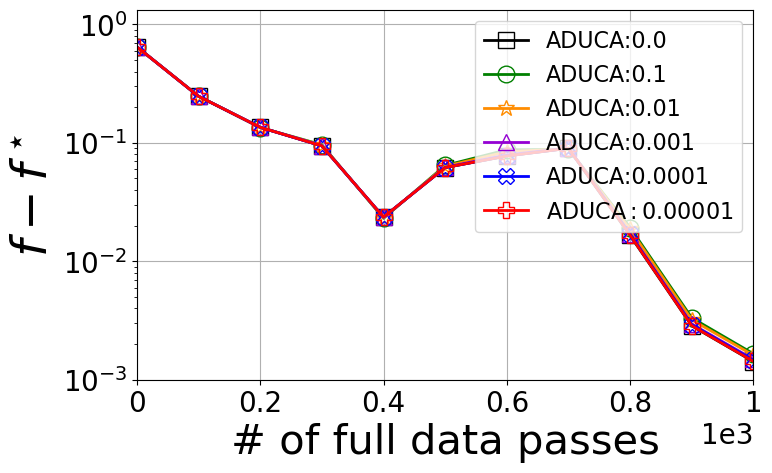}}\hfill
    \hspace*{\fill}
    \subfloat[\texttt{gisette,}\\ $\beta = 0.7,\, \rho = 1.3,\, \gamma = 0.05.$]{\includegraphics[width=0.32\textwidth]{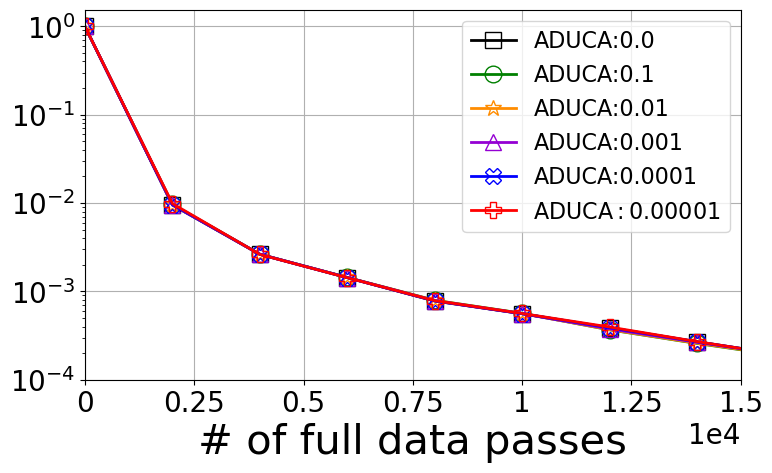}}\hfill
    \hspace*{\fill}
    \subfloat[\texttt{SUSY,}\\  $\beta = 0.7,\, \rho = 1.3,\, \gamma = 0.05.$]{\includegraphics[width=0.32\textwidth]
    {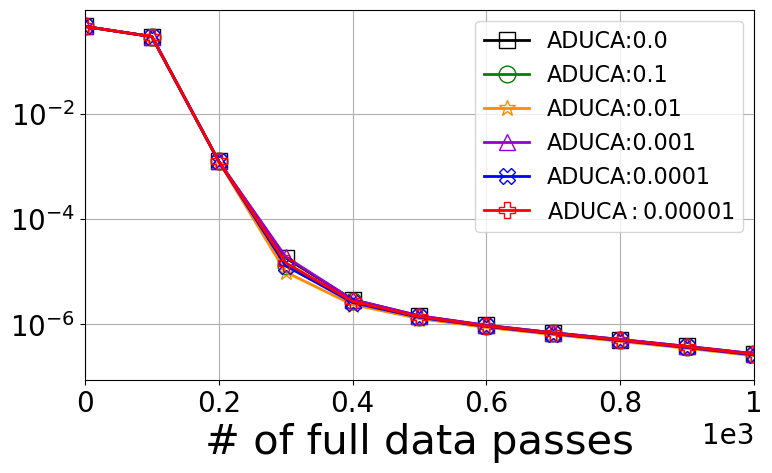}}\hfill
    \hspace*{\fill}
    \subfloat[\texttt{a9a,}\\  $\beta = 0.8,\, \rho = 1.2,\, \gamma = 0.2.$]{\includegraphics[width=0.32\textwidth]{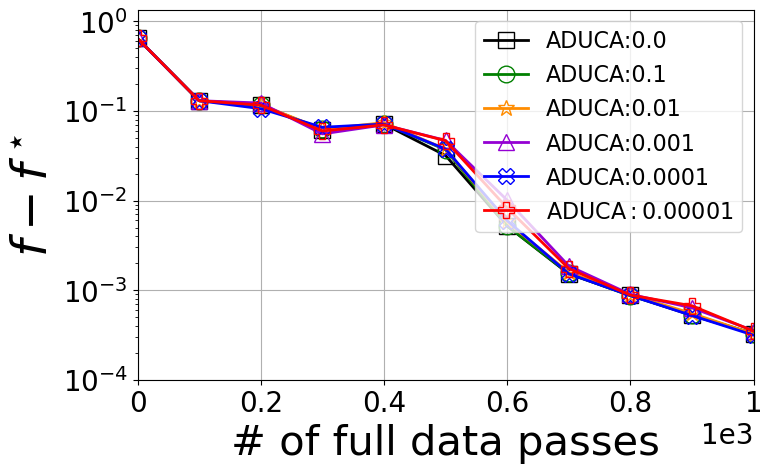}}\hfill
    \hspace*{\fill}
    \subfloat[\texttt{gisette,}\\  $\beta = 0.8,\, \rho = 1.2,\, \gamma = 0.2.$]{\includegraphics[width=0.32\textwidth]{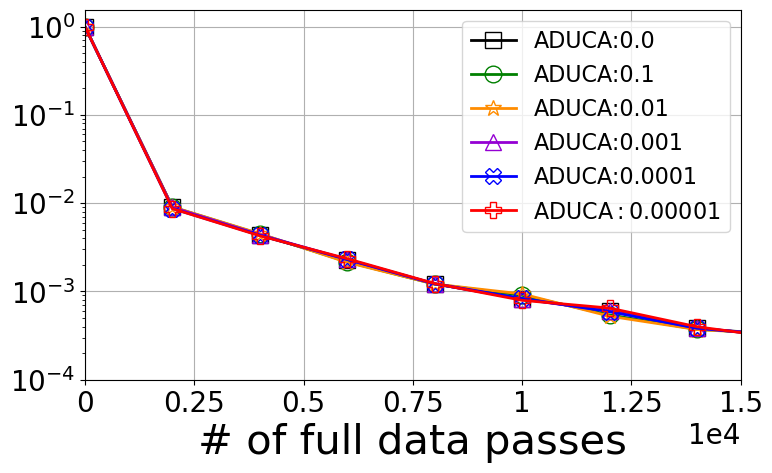}}\hfill
    \hspace*{\fill}
    \subfloat[\texttt{SUSY,}\\  $\beta = 0.8,\, \rho = 1.2,\, \gamma = 0.2.$]{\includegraphics[width=0.32\textwidth]{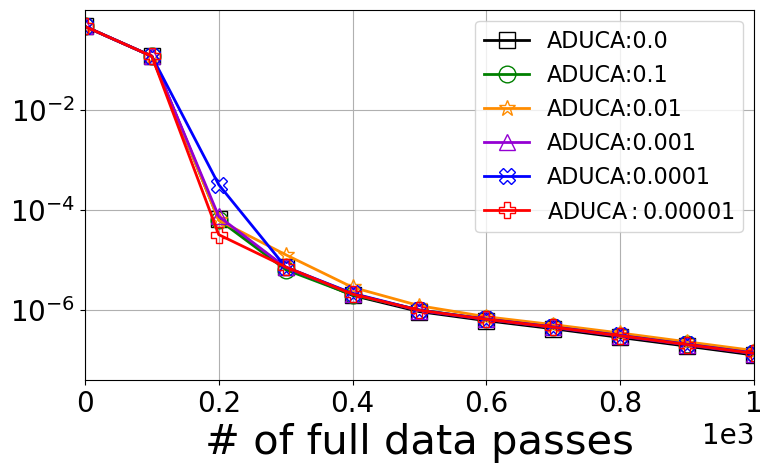}}\hfill
        \hspace*{\fill}
    \subfloat[\texttt{a9a,}\\  $\beta = 0.9,\, \rho = 1.1,\, \gamma = 0.3.$]{\includegraphics[width=0.32\textwidth]{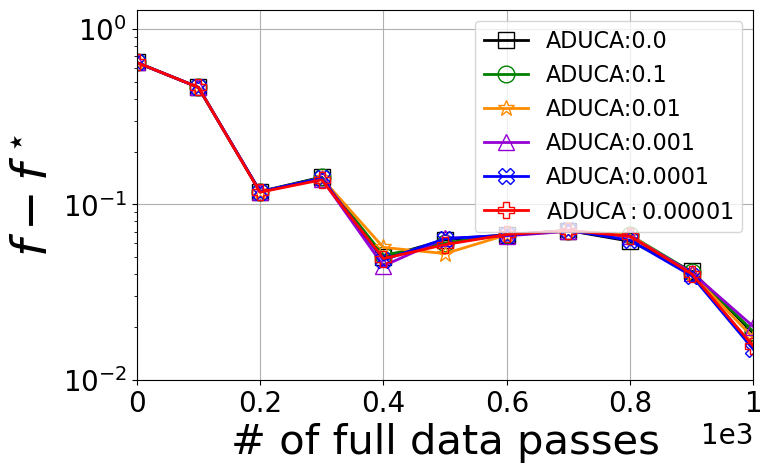}}\hfill
    \hspace*{\fill}
    \subfloat[\texttt{gisette,}\\  $\beta = 0.9,\, \rho = 1.1,\, \gamma = 0.3.$]{\includegraphics[width=0.32\textwidth]{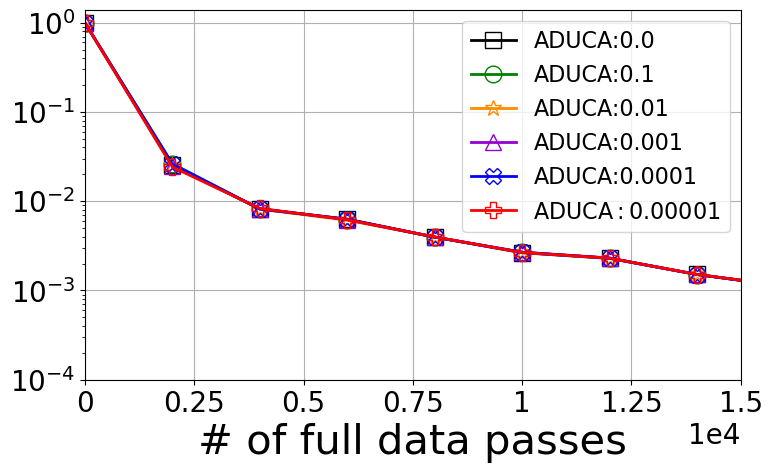}}\hfill
    \hspace*{\fill}
    \subfloat[\texttt{SUSY,}\\  $\beta = 0.9,\, \rho = 1.1,\, \gamma = 0.3.$]{\includegraphics[width=0.32\textwidth]{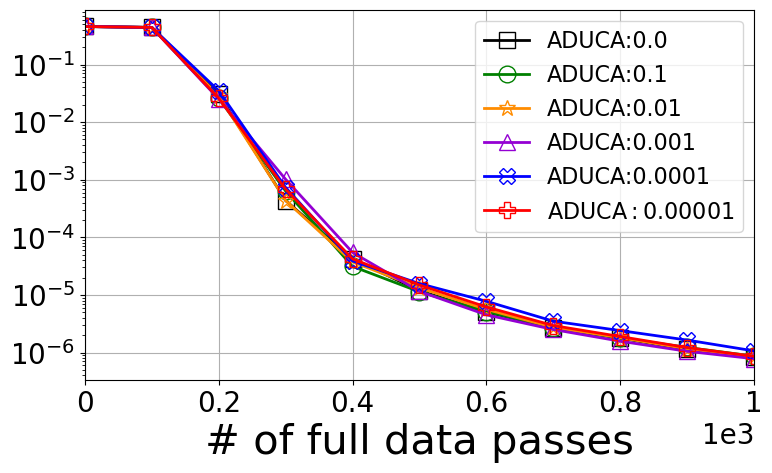}\label{fig:SUSY_ablation_beta-0.9_rho-1.1}}\hfill
    \caption{Sensitivity of \aduca\ to the hyperparameter $\mu$ used in the extrapolation weight $\omega_k$: the numbers in the legend correspond to different choices of $\mu$.}
    \label{fig:ablation}
\end{figure}

In this subsection, we study the sensitivity of \aduca{} to the hyperparameter $\mu$ in \Cref{alg:ADUCA}. Recall that $\mu$ affects the method only through the extrapolation weight $\omega_k = \frac{1 + \rho\beta\mu a_k}{1 + \mu a_k}$, so that $\omega_k = 1$ when $\mu = 0$ and $\omega_k \in (\rho\beta,1)$ when $\mu > 0$. Although $\mu$ is required in our theory to obtain near-linear convergence rates under strong convexity, it is typically unknown in practice. Our ablation study shows that, empirically, this parameter does not meaningfully affect the behavior of \aduca.

We use the SVM setup from~\Cref{subsec:svm} and run \aduca\ with several choices of $\mu$ (reported in the legend), under three different hyperparameter settings $(\beta,\rho,\gamma)$ shown in the subplot titles. As shown in~\Cref{fig:ablation}, the curves corresponding to different $\mu$ values are nearly indistinguishable across datasets and hyperparameter settings. This indicates that the practical performance of \aduca\ is largely insensitive to the attenuation introduced by $\omega_k$; consequently, a rough guess for $\mu$, or simply $\mu = 0$, is sufficient, as shown in our experiments.
 
\section{Conclusion}
We proposed \aduca---a cyclic block-coordinate method for solving generalized monotone variational inequalities. The method combines cyclic updates with fully adaptive step sizes and a one-cycle delayed-update mechanism, which reduces synchronization across blocks and makes \aduca\ naturally suitable for parallel or distributed implementations.
 
\bibliographystyle{plain}
\bibliography{references}

@inproceedings{cai2023empirical,
  author = {Cai, Xufeng and Lin, Cheuk Yin and Diakonikolas, Jelena},
  title = {Tighter convergence bounds for shuffled SGD via primal-dual perspective},
  booktitle = {Proc.~NeurIPS'24},
  year = {2024},
}

@article{nesterov2007,
  author = {Nesterov, Yurii},
  title = {Dual extrapolation and its applications to solving variational inequalities and related problems},
  journal = {Mathematical Programming},
  year = {2007},
  volume = {109},
  number = {2},
  pages = {319--344},
  publisher = {Springer},
}

@article{Nemirovski2004MirrorProx,
  author = {Nemirovski, Arkadi},
  title = {Prox-Method with Rate of Convergence ${O} (1/t)$ for Variational Inequalities with {L}ipschitz Continuous Monotone Operators and Smooth Convex-Concave Saddle Point Problems},
  journal = {SIAM Journal on Optimization},
  year = {2004},
  volume = {15},
  number = {1},
  pages = {229--251},
  publisher = {SIAM},
}

@misc{MathWorksBalance,
  author = {MathWorks},
  title = {Balance -- diagonal scaling to improve eigenvalue accuracy},
  howpublished = {\url{https://www.mathworks.com/help/matlab/ref/balance.html}},
}

@misc{MathWorksEig,
  author = {MathWorks},
  title = {Eig: eigenvalues and eigenvectors},
  howpublished = {\url{https://www.mathworks.com/help/matlab/ref/eig.html}},
}

@misc{RDocumentationBalance,
  author = {RDocumentation},
  title = {Balance a square matrix via {LAPACK}'s dgebal},
  howpublished = {\url{https://www.rdocumentation.org/packages/expm/versions/0.99-1.1/topics/balance}},
}

@misc{NumPyEig,
  author = {NumPy},
  title = {numpy.linalg.eig: Compute the eigenvalues and right eigenvectors of a square array},
  howpublished = {\url{https://numpy.org/doc/stable/reference/generated/numpy.linalg.eig.html}},
}

@misc{JuliaEigen,
  author = {Julia},
  title = {LinearAlgebra.eigen: Compute eigenvalues and eigenvectors},
  howpublished = {\url{https://docs.julialang.org/en/v1/stdlib/LinearAlgebra/\#LinearAlgebra.eigen}},
}

@article{malitsky2020golden,
  author = {Malitsky, Yura},
  title = {Golden ratio algorithms for variational inequalities},
  journal = {Mathematical Programming},
  year = {2020},
  volume = {184},
  number = {1},
  pages = {383--410},
  publisher = {Springer},
}

@article{nesterov2011solving,
  author = {Nesterov, Yurii and Scrimali, Laura},
  title = {Solving strongly monotone variational and quasi-variational inequalities},
  journal = {Discrete and Continuous Dynamical Systems},
  year = {2011},
  volume = {31},
  number = {4},
  pages = {1383--1396},
  publisher = {Discrete and Continuous Dynamical Systems},
}

@book{smith2013matrix,
  author = {Smith, Brian T and Boyle, James M. and Garbow, BS and Ikebe, Y and Klema, VC and Moler, CB},
  title = {Matrix eigensystem routines-{EISPACK} guide},
  series    = {Lecture Notes in Computer Science},
  publisher = {Springer},
  year = {2013},
  volume = {6},
}

@article{Alacaoglu2023BeyondGR,
  author = {Alacaoglu, Ahmet and B{\"o}hm, Axel and Malitsky, Yura},
  title = {Beyond the golden ratio for variational inequality algorithms},
  journal = {Journal of Machine Learning Research},
  volume = {24},
  number = {172},
  pages = {1--33},
  year={2023}
}

@article{diakonikolas2025block,
  title={A Block Coordinate and Variance-Reduced Method for Generalized Variational Inequalities of Minty Type},
  author={Diakonikolas, Jelena},
  journal={ACM/IMS Journal of Data Science},
  volume={2},
  number={2},
  pages={1--30},
  year={2025},
  publisher={ACM New York, NY}
}

@inproceedings{cai2023cyclic,
  author = {Cai, Xufeng and Song, Chaobing and Wright, Stephen and Diakonikolas, Jelena},
  title = {Cyclic block coordinate descent with variance reduction for composite nonconvex optimization},
  booktitle = {Proc.~ICML'23},
  year = {2023},
}

@article{SongDiakonikolas2023CODER,
  author = {Song, Chaobing and Diakonikolas, Jelena},
  title = {Cyclic coordinate dual averaging with extrapolation},
  journal = {SIAM Journal on Optimization},
  year = {2023},
  volume = {33},
  number = {4},
  pages = {2935--2961},
  publisher = {SIAM},
}

@article{osborne1960pre,
  author = {Osborne, EE},
  title = {On pre-conditioning of matrices},
  journal = {Journal of the ACM},
  year = {1960},
  volume = {7},
  number = {4},
  pages = {338--345},
  publisher = {ACM New York, NY, USA},
}

@article{karczmarz1937angenaherte,
  title={Angenaherte auflosung von systemen linearer glei-chungen},
  author={Karczmarz, Stefan},
  journal={Bull. Int. Acad. Pol. Sic. Let., Cl. Sci. Math. Nat.},
  pages={355--357},
  year={1937}
}

@article{duchi2011adaptive,
  title={Adaptive subgradient methods for online learning and stochastic optimization.},
  author={Duchi, John and Hazan, Elad and Singer, Yoram},
  journal={Journal of machine learning research},
  volume={12},
  number={7},
  year={2011}
}

@inproceedings{lin2023accelerated,
  author = {Lin, Cheuk Yin and Song, Chaobing and Diakonikolas, Jelena},
  title = {Accelerated cyclic coordinate dual averaging with extrapolation for composite convex optimization},
  booktitle = {Proc.~ICML'23},
  year = {2023},
}

@article{korpelevich1976extragradient,
  title={The extragradient method for finding saddle points and other problems},
  author={Korpelevich, Galina M},
  journal={Matecon},
  volume={12},
  pages={747--756},
  year={1976}
}

@article{nesterov2012efficiency,
  author = {Nesterov, Yu},
  title = {Efficiency of coordinate descent methods on huge-scale optimization problems},
  journal = {SIAM Journal on Optimization},
  year = {2012},
  volume = {22},
  number = {2},
  pages = {341--362},
  publisher = {SIAM},
}

@inproceedings{diakonikolas2020halpern,
  author = {Diakonikolas, Jelena},
  title = {Halpern Iteration for Near-Optimal and Parameter-Free Monotone Inclusion and Strong Solutions to Variational Inequalities},
  booktitle = {Proc.~COLT'20},
  year = {2020},
}

@article{yousefian2018stochastic,
  author = {Yousefian, Farzad and Nedi{\'c}, Angelia and Shanbhag, Uday V},
  title = {On stochastic mirror-prox algorithms for stochastic {C}artesian variational inequalities: Randomized block coordinate and optimal averaging schemes},
  journal = {Set-Valued and Variational Analysis},
  year = {2018},
  volume = {26},
  pages = {789--819},
  publisher = {Springer},
}

@book{ortega2000iterative,
    author = {Ortega, J. M. and Rheinboldt, W. C.},
    title = {Iterative Solution of Nonlinear Equations in Several Variables},
    publisher = {Society for Industrial and Applied Mathematics},
    year = {2000},
}

@article{chow2017cyclic,
  title={Cyclic coordinate-update algorithms for fixed-point problems: Analysis and applications},
  author={Chow, Yat Tin and Wu, Tianyu and Yin, Wotao},
  journal={SIAM Journal on Scientific Computing},
  volume={39},
  number={4},
  pages={A1280--A1300},
  year={2017},
  publisher={SIAM}
}

@article{beck2013convergence,
  title={On the convergence of block coordinate descent type methods},
  author={Beck, Amir and Tetruashvili, Luba},
  journal={SIAM journal on Optimization},
  volume={23},
  number={4},
  pages={2037--2060},
  year={2013},
  publisher={SIAM}
}

@article{mazumder2011sparsenet,
  author = {Mazumder, Rahul and Friedman, Jerome H and Hastie, Trevor},
  title = {Sparsenet: Coordinate descent with nonconvex penalties},
  journal = {Journal of the American Statistical Association},
  year = {2011},
  volume = {106},
  number = {495},
  pages = {1125--1138},
  publisher = {Taylor \& Francis},
}

@article{friedman2010regularization,
  title={Regularization paths for generalized linear models via coordinate descent},
  author={Friedman, Jerome H and Hastie, Trevor and Tibshirani, Rob},
  journal={Journal of statistical software},
  volume={33},
  pages={1--22},
  year={2010}
}

@inproceedings{gurbuzbalaban2017cyclic,
  author = {G{\"u}rb{\"u}zbalaban, Mert and Ozdaglar, Asuman and Parrilo, Pablo A and Vanli, N Denizcan},
  title = {When cyclic coordinate descent outperforms randomized coordinate descent},
  booktitle = {Proc.~NeurIPS'17},
  year = {2017},
}

@article{sun2019worst,
  author = {Sun, Ruoyu and Ye, Yinyu},
  title = {Worst-case complexity of cyclic coordinate descent: ${O}(n^2)$ gap with randomized version},
  journal = {Mathematical Programming},
  volume={185},
  number={1},
  pages={487--520},
  year={2021},
  publisher={Springer}
}

@inproceedings{song2022coordinate,
  author = {Song, Chaobing and Lin, Cheuk Yin and Wright, Stephen and Diakonikolas, Jelena},
  title = {Coordinate linear variance reduction for generalized linear programming},
  booktitle = {Proc.~NeurIPS'22},
  year = {2022},
}

@article{kotsalis2022simple,
  author = {Kotsalis, Georgios and Lan, Guanghui and Li, Tianjiao},
  title = {Simple and optimal methods for stochastic variational inequalities, {I}: Operator extrapolation},
  journal = {SIAM Journal on Optimization},
  year = {2022},
  volume = {32},
  number = {3},
  pages = {2041--2073},
  publisher = {SIAM},
}

@article{ouyang2021lower,
  author = {Ouyang, Yuyuan and Xu, Yangyang},
  title = {Lower complexity bounds of first-order methods for convex-concave bilinear saddle-point problems},
  journal = {Mathematical Programming},
  year = {2021},
  volume = {185},
  number = {1-2},
  pages = {1--35},
  publisher = {Springer},
}

@inproceedings{golowich2020last,
  author = {Golowich, Noah and Pattathil, Sarath and Daskalakis, Constantinos and Ozdaglar, Asuman},
  title = {Last iterate is slower than averaged iterate in smooth convex-concave saddle point problems},
  booktitle = {Proc.~COLT'20},
  year = {2020},
}

@inproceedings{diakonikolas2021efficient,
  author = {Diakonikolas, Jelena and Daskalakis, Constantinos and Jordan, Michael I},
  title = {Efficient methods for structured nonconvex-nonconcave min-max optimization},
  booktitle = {Proc.~AISTATS'21},
  year = {2021},
}

@article{diakonikolas2021potential,
  title={Potential function-based framework for minimizing gradients in convex and min-max optimization},
  author={Diakonikolas, Jelena and Wang, Puqian},
  journal={SIAM Journal on Optimization},
  volume={32},
  number={3},
  pages={1668--1697},
  year={2022},
  publisher={SIAM}
}

@inproceedings{malitsky2019adaptive,
  author = {Malitsky, Yura and Mishchenko, Konstantin},
  title = {Adaptive gradient descent without descent},
  booktitle = {Proc.~ICML'20},
  year = {2020},
}

@inproceedings{malitsky2024adaptive,
  author = {Malitsky, Yura and Mishchenko, Konstantin},
  title = {Adaptive proximal gradient method for convex optimization},
  booktitle = {Proc.~NeurIPS'24},
  year = {2024},
}

@article{latafat2024adaptive,
  author = {Latafat, Puya and Themelis, Andreas and Stella, Lorenzo and Patrinos, Panagiotis},
  title = {Adaptive proximal algorithms for convex optimization under local {L}ipschitz continuity of the gradient},
  journal = {Mathematical Programming},
  volume={213},
  number={1},
  pages={433--471},
  year={2025},
  publisher={Springer}
}

@article{latafat2025convergence,
  author = {Latafat, Puya and Themelis, Andreas and Villa, Silvia and Patrinos, Panagiotis},
  title = {On the convergence of proximal gradient methods for convex simple bilevel optimization},
  journal = {Journal of Optimization Theory and Applications},
  volume={204},
  number={3},
  pages={51},
  year={2025},
  publisher={Springer}
}

@article{li2023simple,
  title={A simple uniformly optimal method without line search for convex optimization},
  author={Li, Tianjiao and Lan, Guanghui},
  journal={Mathematical Programming},
  year={2025},
  publisher={Springer}
}

@inproceedings{vladarean2021first,
  author = {Vladarean, Maria-Luiza and Malitsky, Yura and Cevher, Volkan},
  title = {A first-order primal-dual method with adaptivity to local smoothness},
  booktitle = {Proc.~NeurIPS'21},
  year = {2021},
}

@article{lan2024auto,
  author = {Lan, Guanghui and Li, Tianjiao},
  title = {Auto-conditioned primal-dual hybrid gradient method and alternating direction method of multipliers},
  journal = {arXiv preprint arXiv:2410.01979},
  year = {2024},
}

@article{lan2024projected,
  author = {Lan, Guanghui and Li, Tianjiao and Xu, Yangyang},
  title = {Projected gradient methods for nonconvex and stochastic optimization: New complexities and auto-conditioned stepsizes},
  journal = {arXiv preprint arXiv:2412.14291},
  year = {2024},
}

@article{chang2011libsvm,
  author = {Chang, Chih-Chung and Lin, Chih-Jen},
  title = {{LIBSVM}: a library for support vector machines},
  journal = {ACM Transactions on Intelligent Systems and Technology},
  year = {2011},
  volume = {2},
  number = {3},
  pages = {1--27},
  publisher = {Acm New York, NY, USA},
}

@inproceedings{
borodich2026nesterov,
title={Nesterov Finds {GRAAL}: Optimal and Adaptive Gradient Method for Convex Optimization},
author={Ekaterina Borodich and Dmitry Kovalev},
booktitle={Proc.~ICLR'26},
year={2026},
}

@article{suh2025adaptive,
  author = {Suh, Jaewook J and Ma, Shiqian},
  title = {An Adaptive and Parameter-Free Nesterov's Accelerated Gradient Method for Convex Optimization},
  journal = {arXiv preprint arXiv:2505.11670},
  year = {2025},
}

@inproceedings{bach2019universal,
  title={A universal algorithm for variational inequalities adaptive to smoothness and noise},
  author={Bach, Francis and Levy, Kfir Y},
  booktitle={Proc.~COLT'19},
  year={2019},
}

@inproceedings{DBLP:conf/aaai/EneNV21,
  author = {Alina Ene and Huy L. Nguyen and Adrian Vladu},
  title = {Adaptive Gradient Methods for Constrained Convex Optimization and Variational Inequalities},
  booktitle = {Proc.~AAAI'21},
  year = {2021},
}

@inproceedings{DBLP:conf/aaai/EneN22,
  author = {Alina Ene and Huy Le Nguyen},
  title = {Adaptive and Universal Algorithms for Variational Inequalities with Optimal Convergence},
  booktitle = {Proc.~AAAI'22},
  year = {2022},
}

\appendix

\section{Proof of Lemma \ref{lemma:simplified-steps}}

\begin{proof}
Observe first that the last condition for $k \geq 1$ in \eqref{eq:complete-step-size-conditions}  implies a lower bound on $\phi_{k-1}.$ Specifically, as $a_k \leq \frac{\eta}{2} \frac{\sqrt{(1-\tau)\omega_{k-1}/2}}{\hat{L}_k}\phi_{k-1}$, we have
\begin{equation}\label{eq:lower-bnd-phi-k-1}
    \phi_{k-1} \geq a_k \hat{L}_k \frac{2\sqrt{2}}{\eta \sqrt{(1-\tau)\omega_{k-1}}},\quad \forall k \geq 1. 
\end{equation}
We now argue, using \eqref{eq:lower-bnd-phi-k-1}, that the right-hand sides of inequalities in the last two lines of \eqref{eq:complete-step-size-conditions} are bounded below by one of the first two conditions in \eqref{eq:complete-step-size-conditions}. First, \begin{align}
    \frac{\eta \sqrt{\tau}}{2\sqrt{3}}\frac{\omega_{k-1}\sqrt{\omega_{k-2}}}{\hat{L}_{k-1}\sqrt{1 + \omega_{k-1}}}\phi_{k-2} &\geq \sqrt{\frac{2}{3}}\cdot \sqrt{\frac{\tau}{1-\tau}} \cdot \frac{\omega_{k-1}}{\sqrt{1+\omega_{k-1}}} a_{k-1}\notag\\
    &\geq \sqrt{\frac{2}{3}}\cdot \sqrt{\frac{\tau}{1-\tau}} \cdot \frac{\rho\beta}{\sqrt{1+\rho\beta}} a_{k-1},\notag
\end{align}
where the last inequality comes from $\omega_{k-1} \in (\rho\beta, 1].$ The last inequality is bounded below by $\rho_0 a_{k-1}$ whenever $\sqrt{\frac{2}{3}}\cdot \sqrt{\frac{\tau}{1-\tau}} \cdot \frac{\rho\beta}{\sqrt{1+\rho\beta}} \geq \rho_0,$ which, solving for $\tau,$ is equivalent to the condition $\tau \geq \frac{3\rho_0^2(1 + \rho\beta)}{2(\rho\beta)^2 + 3\rho_0^2(1 + \rho\beta)}.$

For the remaining condition in \eqref{eq:complete-step-size-conditions}, we similarly argue that under an appropriate choice of $\tau,$ the right-hand side of the last inequality in \eqref{eq:complete-step-size-conditions} is bounded below by $\rho_0 a_{k-1},$ so this condition must be satisfied as $a_k \leq \rho_0 a_{k-1}$ by the first condition in \eqref{eq:complete-step-size-conditions}. To see this, again using \eqref{eq:lower-bnd-phi-k-1}, we have
\begin{align*}
    \frac{\eta\sqrt{\tau}}{2\sqrt{3}}\frac{\omega_{k-1}}{\hat{L}_{k-2}\sqrt{1 + \omega_{k-1}}}\sqrt{\frac{\omega_{k-3}}{\omega_{k-2}}}\frac{a_{k-1}}{a_{k-2}}\phi_{k-3} &\geq \sqrt{\frac{2}{3}}\cdot \sqrt{\frac{\tau}{1-\tau}} \cdot \frac{\omega_{k-1}}{\sqrt{\omega_{k-2}(1+\omega_{k-1})}} a_{k-1} \\
    &\geq \sqrt{\frac{2}{3}}\cdot \sqrt{\frac{\tau}{1-\tau}}\cdot \frac{\rho \beta}{\sqrt{1+\rho\beta}}a_{k-1},  
\end{align*}
where we used that $\omega_{k-2} \leq 1$ and $\frac{\omega_{k-1}}{\sqrt{1+\omega_{k-1}}} \geq \frac{\rho \beta}{\sqrt{1+\rho\beta}}.$ As already argued, the last inequality is further bounded below by $\rho_0 a_{k-1}$ for $\tau \geq \frac{3\rho_0^2(1 + \rho\beta)}{2(\rho\beta)^2 + 3\rho_0^2(1 + \rho\beta)}.$

Setting $\tau = \frac{3\rho_0^2(1 + \rho\beta)}{2(\rho\beta)^2 + 3\rho_0^2(1 + \rho\beta)} < 1,$ \eqref{eq:complete-step-size-conditions} simplifies to requiring, for $k \geq 1,$ that
\begin{equation}\notag \begin{aligned}
        a_k &\leq \min\Big\{\rho_0 a_{k-1},  \frac{\eta}{2}\min\Big\{ \Big( \frac{\sqrt{\tau}\, \omega_{k-1}}{\sqrt{3}L_k\sqrt{1+\omega_{k-1}}},\, \frac{\sqrt{(1-\tau)\omega_{k-1}}}{\sqrt{2}\hat{L}_k}\Big)\Big\}\phi_{k-1}\Big\}.\end{aligned}
\end{equation}
Notice that because both $\omega_{k-1}$ and $\phi_{k-1}$ depend on $\mu,$ this choice of step sizes is valid only for the known value of $\mu.$ When $\mu$ is not known, we can simply use $\mu \geq 0$ and $\omega_{k-1} \in (\rho\beta, 1]$ to bound below $\omega_{k-1}, \frac{\omega_{k-1}}{\sqrt{1+\omega_{k-1}}},$ and $\phi_{k-1}$, and obtain the  sufficient conditions for the step size stated in \eqref{eq:simple-step-size-conditions-unknown-mu}. 
\end{proof}

\end{document}